\documentclass[12pt,a4paper,english]{article}
\usepackage[utf8]{inputenc}
\usepackage{babel, fullpage, amsmath, amsthm, amsfonts, amssymb, xcolor, dsfont}
\usepackage{graphicx}
\usepackage{caption}
\newtheorem{theorem}{Theorem}[section]
\newtheorem{corollary}[theorem]{Corollary}
\newtheorem{proposition}[theorem]{Proposition}
\newtheorem{lemma}[theorem]{Lemma}
\newtheorem{lemma-annexe}{Lemma}
\theoremstyle{definition}

\newtheorem{remark}{Remark}[section]
\newcommand{\keywords}{\\

{\bf Keywords : }}
\newcommand{\subclass}{\\

{\bf Mathematics Subject Classification (2000) : }}
\newcommand{\dsp}{\displaystyle}

\newcommand{\Lb}{\mathbb{L}}

\newcommand{\Pz}{\left[\begin{matrix}p_0 \\ p_1\end{matrix}\right]}

\newcommand{\T}{\mathbb{T}}

\newcommand{\D}{\mathcal{D}\left(A\right)}
\newcommand{\Ds}{\mathcal{D}\left(A^2\right)}
\newcommand{\Dz}{\mathcal{D}\left(A_0\right)}
\newcommand{\Dzs}{\mathcal{D}\left(A_0^2\right)}
\newcommand{\Dzdemi}{\mathcal{D}\left(A_0^\frac{1}{2}\right)}
\newcommand{\Dztdemi}{\mathcal{D}\left(A_0^\frac{3}{2}\right)}
\newcommand{\FORALL} {{\hbox{$\hskip 11mm \forall \;$}}}
\author{Ghislain HAINE\\
{\small Universit\'e Henri Poincar\'e (Institut \'Elie Cartan)}\\
{\small B.P. 70239, 54506 Vandoeuvre-les-Nancy, France}\\
{\small Ghislain.Haine@iecn.u-nancy.fr}
\and Karim RAMDANI\\
{\small INRIA Nancy Grand-Est (CORIDA)}\\
{\small 615 rue du Jardin Botanique, 54600 Villers-les-Nancy, France}\\
{\small Karim.Ramdani@inria.fr}
}
\date{}
\title{\textbf{Reconstructing initial data using observers : error analysis of the semi-discrete and fully discrete approximations.}}
\makeindex
\begin{document}\maketitle

\begin{abstract}
A new iterative algorithm for solving initial data inverse problems from partial observations has been recently proposed in Ramdani, Tucsnak and Weiss \cite{RamTucWei10}. Based on the concept of observers  (also called Luenberger observers), this algorithm covers a large class of abstract evolution PDE's. In this paper, we are concerned with the convergence analysis of this algorithm. More precisely, we provide a complete numerical analysis for semi-discrete (in space) and fully discrete approximations derived using finite elements in space and finite differences in time. The analysis is carried out for abstract Schr\"odinger and wave conservative systems with bounded observation (locally distributed).
\keywords Control theory - Observers - Inverse problems - Finite element and finite difference discretization - Convergence analysis - Schr\"odinger equation - Wave equation
\subclass Primary : 35Q93 \\ Secondary : 35L05 - 35J10 - 65M22
\end{abstract}
%
\section{Introduction}
\setcounter{equation}{0}
\label{Sect_Intro}
%
The goal of this paper is to present a convergence analysis for the iterative algorithm recently proposed in Ramdani, Tucsnak and Weiss \cite{RamTucWei10} for solving initial state inverse problems from measurements over a time interval. This algorithm is based on the use back and forth in time of observers (sometimes called Luenberger observers or Kalman observers; see for instance Curtain and Zwart \cite{CurZwa95}). Let us emphasize that during the last decade, observers have been designed for linear and nonlinear infinite-dimensional systems in many works, among which we can mention for instance Auroux and Blum \cite{AurBlu08} in the context of data assimilation, Deguenon, Sallet and Xu \cite{DegSalXu06}, Guo and Guo \cite{GuoGuo09}, Guo and Shao \cite{GuoSha09} in the context of wave-type systems, Lasiecka and Triggiani \cite{LasTri00}, Smyshlyaev and Krstic \cite{SmyKrs05} for parabolic systems and Krstic, Magnis and Vazquez \cite{KrsMagVaz09} for the non linear viscous Burgers equation.

Let us first briefly describe the principle of the reconstruction method proposed in \cite{RamTucWei10} in the simplified context of skew-adjoint generators and bounded observation operator. We will always work under these assumptions throughout the paper. Given two Hilbert spaces $X$ and $Y$ (called \emph{state} and \emph{output} spaces respectively), let $A : \D \rightarrow X$ be skew-adjoint operator generating a $C_0$-group $\T$ of isometries on $X$ and let $C \in \mathcal{L}(X,Y)$ be a bounded observation operator. Consider the infinite dimensional linear system given by 
\begin{equation}\label{eq-sys-initial}
\left\lbrace \begin{array}{ll}
\dot{z}(t) = Az(t), &\quad \forall t \geqslant 0, \\
y(t) = Cz(t), &\quad \forall t \in [0,\tau].
\end{array} \right.
\end{equation}
where $z$ is the state and $y$ the output function (throughout the paper, the dot symbol is used to denote the time derivative).  Such systems are often used as models of vibrating systems (e.g., the wave equation, the beam equation,...), electromagnetic phenomena (Maxwell's equations) or in quantum mechanics (Schr\"odinger's equation).
\begin{center}
	\includegraphics[scale=0.5]{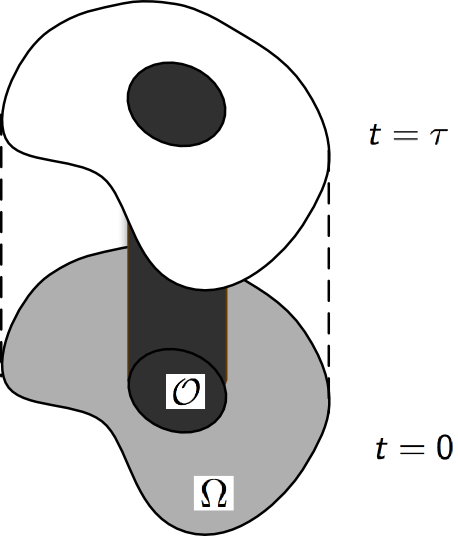}
	\captionof{figure}{An initial data inverse problem for evolution PDE's : How to reconstruct the initial state (light grey) for a PDE set on a domain $\Omega$ from partial observation on ${\cal O}\times [0,\tau]$ (dark grey)?}
	\label{Fig1}
\end{center}

The inverse problem considered here is to reconstruct the initial state $z_0 =z(0)$ of system \eqref{eq-sys-initial} knowing (\emph{the observation}) $y(t)$ on the time interval $[0, \tau]$ (see Fig. \ref{Fig1}). Such inverse problems arise in many applications, like thermoacoustic tomography Kuchment and Kunyansky \cite{KucKun08} or data assimilation Puel \cite{Pue09}. To solve this inverse problem, we assume here that it is well-posed, \emph{i.e.} that $(A,C)$ is exactly observable in time $\tau > 0$, i.e. that there exists $k_\tau>0$ such that
$$ 
\int_0^\tau \|y(t)\|^2 dt\geq k_\tau^2 \|z_0\|^2, \FORALL z_0\in{\cal D}(A).
$$
Following Liu \cite[Theorem 2.3.]{Liu97}, we know that $A^+ = A-C^*C$ (respectively $A^- = -A-C^*C$) generate an exponentially stable $C_0$-semigroup $\T^+$ (respectively $\T^-$) on $X$. Then, we introduce the following initial and final Cauchy problems, called respectively \emph{forward} and \emph{backward observers} of \eqref{eq-sys-initial}
\begin{equation}\label{eq-sys-obs-forward}
\left\lbrace \begin{array}{ll}
\dot{z}^+(t) = A^+z^+(t) + C^*y(t), \quad \forall t \in [0,\tau], \\
z^+(0) = 0,
\end{array} \right.
\end{equation}
\begin{equation}\label{eq-sys-obs-backward}
\left\lbrace \begin{array}{ll}
\dot{z}^-(t) = -A^-z^-(t) - C^*y(t), \quad \forall t \in [0,\tau], \\
z^-(\tau) = z^+(\tau).
\end{array} \right.
\end{equation}
Note that the states $z^+$ and $z^-$ of the forward and backward observers are completely determined by the knowledge of the output $y$. If we set $\Lb_\tau = \T_\tau^-\T_\tau^+$, then by \cite[Proposition 3.7]{RamTucWei10}, we have $\eta:=\|\Lb_\tau\|_{\mathcal{L}(X)} < 1$ and by \cite[Proposition 3.3]{RamTucWei10}, the following remarkable relation holds true
\begin{equation}\label{eq-z0-Ltau}
z_0 = (I-\Lb_\tau)^{-1}z^-(0).
\end{equation}
In particular, one can invert the operator $(I-\Lb_\tau)$ using a Neumann series and get the following expression for the initial state
\begin{equation}\label{eq-z0-neumann}
z_0 = \sum_{n=0}^\infty \Lb_\tau^{n}z^-(0).
\end{equation}
Thus, at least theoretically, the reconstruction of the initial state is given by the above formula. Note that  the computation of each term in the above sum requires to solve the two non-homogeneous systems \eqref{eq-sys-obs-forward} and \eqref{eq-sys-obs-backward}. In practice, the reconstruction procedure requires the discretization of these two systems and the truncation of the infinite sum in \eqref{eq-z0-neumann} to keep only a finite number of back and forth iterations. For instance, if we consider a space semi-discretization corresponding to a mesh size $h$ (typically a finite element approximation), one can only compute 
\begin{equation}\label{eq-z0h-neumann}
z_{0,h} = \sum_{n=0}^{N_h} \Lb_{h,\tau}^{n}z_h^-(0),
\end{equation}
where
\begin{itemize}
\item
$\Lb_{h,\tau}=\T_{h,\tau}^-\T_{h,\tau}^+$, where $\T_{h,\tau}^\pm\in {\cal L}(X)$ are suitable space discretizations of $\T_{\tau}^\pm$,
\item
$z_h^-(0)\in X_h$ is an approximation of $z^-(0)$ in a suitable finite dimensional subspace $X_h$ of $X$,
\item
$N_h$ is a suitable truncation parameter.
\end{itemize}
Similarly, if a full discretization described by a mesh size $h$ and a time step $\Delta t$ is considered, one can compute 
\begin{equation}\label{eq-z0hdelta-neumann}
z_{0,h,\Delta t} = \sum_{n=0}^{N_{h,\Delta t}} \Lb_{h,\Delta t,K}^{n} \left(z^-_{h}\right)^0.
\end{equation}
where
\begin{itemize}
\item
$ \Lb_{h,\Delta t,K}=\T_{h,\Delta t,K}^-\T_{h,\Delta t,K}^+$, where $\T_{h,\Delta t,K}^\pm$ are suitable space and time discretizations of $\T_{\tau}^\pm$,
\item
$\left(z^-_h\right)^0\in X_h$ is an approximation of $z^-(0)$,
\item
$N_{h,\Delta t}$ is a suitable truncation parameter.
\end{itemize}
For the sake of clarity, the precise definitions of the spaces and discretizations used will be given later in the paper.

Our objective in this work is to propose a convergence analysis of $z_{0,h}$ and $z_{0,h,\Delta t} $ towards $z_0$. A particular attention will be devoted to the optimal choice of the truncation parameters $N_{h}$ and $N_{h,\Delta t}$ for given discretization parameters (mesh size $h$ and time step $\Delta t$). Let us emphasize that our error estimates (see \eqref{eqz0z0h}, \eqref{eqz0z0hk}, \eqref{eq-w0w0h-w1w1h} and \eqref{eq-w0w0hk-w1w1hk}) provide in particular an upper bound for the maximum admissible noise under which convergence of the algorithm is guaranteed. As usually in approximation error theory of PDE's, some regularity assumptions are needed to obtain our error estimates. Namely, our result allows us to reconstruct only initial data contained in some subspace of $X$ (namely $\Ds$). Moreover, our analysis only holds for locally distributed observation (leading to bounded observation operators).

Throughout the paper, we denote by $M$ a constant independent of $\tau$, of the initial state $z_0$ and of the discretization parameters $h$ and $\Delta t$, but which may differ from line to line in the computations. 

The paper is organized as follows: in Section \ref{Sect_Schr} we provide a convergence analysis of the algorithm for an abstract Schr\"odinger type system, by considering successively the semi-discretization (Subsection \ref{Subsct_Schr1}) and the full discretization (Subsection \ref{Subsct_Schr2}). In Section \ref{Sect_Wave}, similar results are given for an abstract wave system. Once again, we tackle successively the semi-discretization (Subsection \ref{Subsct_Wave1}) and the full discretization (Subsection \ref{Subsct_Wave2}). Finally, the Appendix is devoted to the proof of two technical lemmas which are used several times troughout the paper.

\section{Schr\"odinger equation}
\setcounter{equation}{0}
\label{Sect_Schr}
Let $X$ be a Hilbert space endowed with the inner product $\left\langle\cdot,\cdot\right\rangle$. Let $A_0 : \Dz \rightarrow X$ be a strictly positive self-adjoint operator and $C \in \mathcal{L}(X,Y)$ a bounded observation operator, where $Y$ is an other Hilbert space. The norm in ${\cal D}(A_0^\alpha)$ will be denoted by $\|\cdot\|_\alpha$.
We assume that there exists some $\tau>0$ such that $(iA_0,C)$ is exactly observable in time $\tau$. Thus by Liu \cite[Theorem 2.3.]{Liu97}, $A^+=iA_0-C^*C$ (resp. $A^-=-iA_0-C^*C$) is the generator of an exponentially stable $C_0$-semigroup $\T^+$ (resp. $\T^-$). We want to reconstruct the initial value $z_0$ of the following system 
\begin{equation}\label{eq-sys-initial-schrodinger}
\left\lbrace \begin{array}{ll}
\dot{z}(t) = iA_0 z(t), \quad \forall t \geqslant 0, \\
y(t) = C z(t), \quad \forall t \in [0,\tau].
\end{array} \right.
\end{equation}
Throughout this section we always assume that $z_0 \in \Dzs$. Thus by applying Theorem 4.1.6 of Tucsnak and Weiss \cite{TucWei09}, we have $$z \in C\left([0,\tau],\Dzs\right) \cap C^1\left([0,\tau],\Dz\right).$$
The forward and backward observers \eqref{eq-sys-obs-forward} and \eqref{eq-sys-obs-backward} read then as follows 
\begin{equation}\label{eq-sys-obs-forward-schrodinger}
\left\lbrace \begin{array}{ll}
\dot{z}^+(t) = iA_0 z^+(t) - C^*C z^+(t) + C^*y(t), \quad \forall t \in [0,\tau], \\
z^+(0) = 0,
\end{array} \right.
\end{equation}
\begin{equation}\label{eq-sys-obs-backward-schrodinger}
\left\lbrace \begin{array}{ll}
\dot{z}^-(t) = iA_0 z^-(t) + C^*C z^-(t) - C^*y(t), \quad \forall t \in [0,\tau], \\
z^-(\tau) = z^+(\tau).
\end{array} \right.
\end{equation}
Clearly, the above systems can be rewritten in the general form of an initial value Cauchy problem (simply by using a time reversal for the second system)
\begin{equation}\label{eq-general-q-sch}
\left\{\begin{array}{ll}
\dot{q}(t) = \pm i A_0q(t) - C^*C q(t) + F(t), \qquad \forall t \in [0,\tau],\\
q(0) = q_0,
\end{array}\right.
\end{equation}
where we have set 
\begin{itemize}
\item for the forward observer \eqref{eq-sys-obs-forward-schrodinger} : $F(t) = C^*y (t)= C^*C z(t)$ and $q_0=0$,
\item for the backward observer \eqref{eq-sys-obs-backward-schrodinger} : $F(t) = C^*y(\tau-t) = C^*C z(\tau-t)$ and $q_0=z^+(\tau)\in \Dzs$.
\end{itemize}


\subsection{Space Semi-Discretization}
\label{Subsct_Schr1}

\subsubsection{Statement of the main result}
We use a Galerkin method to approximate system \eqref{eq-general-q-sch}. More precisely, consider a family $(X_h)_{h>0}$ of finite-dimensional subspaces of $\Dzdemi$ endowed with the norm in $X$. We denote $\pi_h$ the orthogonal projection from $\Dzdemi$ onto $X_h$. We assume that there exist $M>0$, $\theta>0$ and $h^*>0$ such that we have for all $h\in(0,h^*)$ 
\begin{equation}\label{eq-est-X}
\left\| \pi_h\varphi-\varphi \right\| \leq Mh^\theta \left\| \varphi \right\|_\frac{1}{2}, \qquad \forall \varphi \in \Dzdemi.
\end{equation}
Given $q_0 \in \Dzs$, the variational formulation of \eqref{eq-general-q-sch} reads for all $t\in[0,\tau]$ and all $\varphi \in \Dzdemi$ as follows
\begin{equation}\label{eq-general-q-FV-sch}
\left\{\begin{array}{ll}
\left\langle \dot{q}(t),\varphi \right\rangle = \pm i\left\langle q(t),\varphi \right\rangle_\frac{1}{2} - \left\langle C^*Cq(t),\varphi \right\rangle + \left\langle F(t),\varphi \right\rangle,\\
q(0)=q_0.
\end{array}\right.
\end{equation}
Suppose that $q_{0,h}\in X_h$ and $F_h$ are given approximations of $q_{0}$ and $F$ respectively in the spaces $X$ and $L^1\left([0,\tau],X\right)$. For all $t\in [0,\tau]$, we define $q_h(t)\in X_h$ as the unique solution of the variational problem 
\begin{equation}\label{eq-general-qh-FV-sch}
\left\{\begin{array}{ll}
\left\langle \dot{q}_h(t),\varphi_h\right\rangle = \pm i\left\langle q_h(t),\varphi_h \right\rangle_\frac{1}{2} - \left\langle C^*Cq_h(t),\varphi_h \right\rangle + \left\langle F_h(t),\varphi_h \right\rangle,\\
q_h(0)=q_{0,h}.
\end{array}\right.
\end{equation}
for all $\varphi_h \in X_h$.

The above approximation procedure leads in particular to the definition of the semi-discretized versions $\T^\pm_h$ of the semigroups $\T^\pm$ that we will use. Indeed, we simply set 
$$\T^+_t q_0\simeq \T^+_{h,t} q_0=q_h(t)\qquad \qquad 
\T^-_t q_0\simeq \T^-_{h,t} q_0=q_h(\tau-t)$$
where $q_h$ is the solution of equation \eqref{eq-general-qh-FV-sch} with the corresponding sign and for $F_h=0$ and $q_{0,h}=\pi_hq_0$.  The approximation of $\Lb_\tau = \T_\tau^-\T_\tau^+$ follows immediately by setting
$$
\Lb_{h,\tau} = \T^-_{h,\tau}\T^+_{h,\tau}.
$$ 

Assume that $y_h$ is an approximation of the output $y$ in $L^1([0,\tau],Y)$ and let ${z}_h^+$ and ${z}_h^-$ denote the Galerkin approximations of the solutions of systems \eqref{eq-sys-obs-forward-schrodinger} and \eqref{eq-sys-obs-backward-schrodinger}, satisfying for all $t \in [0,\tau]$ and all $\varphi_h \in X_h$
$$
\left\{\begin{array}{ll}
\left\langle \dot{z}_h^+(t),\varphi_h\right\rangle = i\left\langle {z}_h^+(t),\varphi_h \right\rangle_\frac{1}{2} - \left\langle C^*C{z}_h^+(t),\varphi_h \right\rangle + \left\langle C^*y_h(t),\varphi_h \right\rangle,\\
 {z}_h^+(0)=0.
\end{array}\right.
$$
$$
\left\{\begin{array}{ll}
\left\langle \dot{z}_h^-(t),\varphi_h\right\rangle = i\left\langle {z}_h^-(t),\varphi_h \right\rangle_\frac{1}{2} + \left\langle C^*C{z}_h^-(t),\varphi_h \right\rangle - \left\langle C^*y_h(t),\varphi_h \right\rangle,\\
 {z}_h^-(\tau)= {z}_h^+(\tau).
\end{array}\right.
$$

Thus, our main result in this subsection reads as follows.
\begin{theorem}\label{th-main-sch}
Let $A_0 : \Dz \rightarrow X$ be a strictly positive self-adjoint operator and $C \in \mathcal{L}(X,Y)$ such that $C^*C \in \mathcal{L}\left(\Dzs\right)\cap\mathcal{L}\left(\Dz\right)$. Assume that the pair $(iA_0,C)$ is exactly observable in time $\tau>0$ and set $\eta:=\|\Lb_\tau\|_{\mathcal{L}(X)} < 1$. Let $z_0\in \Dzs$ be the initial value of \eqref{eq-sys-initial-schrodinger} and $z_{0,h}$ be defined by \eqref{eq-z0h-neumann}. 

Then there exist $M > 0$ and $h^* > 0$ such that for all $h \in (0, h^*)$ 
$$
\| z_0 - z_{0,h} \| \leq M \left[ \left(\frac{\eta^{N_h+1}}{1-\eta} + h^\theta \tau N_h^2 \right)\| z_0 \|_2 + N_h \int_0^\tau \| C^*\left(y(s)-y_h(s)\right) \|ds \right].
$$
\end{theorem}
A particular choice of $ N_h $ leads to an explicit error estimate (with respect to $h$) as shown in the next Corollary (the proof is left to the reader because of its simplicity) 
\begin{corollary}
Under the assumptions of Theorem \ref{th-main-sch}, we set 
$$
N_h = \theta \frac{\ln h}{\ln \eta}.
$$
Then, there exist $ M_\tau > 0 $ and $ h^* > 0 $ such that for all $h \in (0, h^*)$
\begin{equation}\label{eqz0z0h}
\| z_0 - z_{0,h} \| \leq M_\tau \left(h^\theta \ln^2 h\, \| z_0 \|_2 + |\ln h| \int_0^\tau \| C^*\left(y(s)-y_h(s)\right) \|ds \right).
\end{equation}
\end{corollary}

\begin{remark}
In fact, Theorem \ref{th-main-sch} still holds true for $z_0\in\Dztdemi$ (with the same proofs and slightly adapting the spaces). Nevertheless, we have not been able to carry out this analysis for the fully discrete approximation in this case. This is why we restricted our analysis to the case of an initial data $z_0\in\Dzs$.
\end{remark}

\subsubsection{Proof of Theorem \ref{th-main-sch}}
Before proving Theorem \ref{th-main-sch}, we first need to prove some auxiliary results. The next Proposition, which constitutes one of the main ingredients of the proof, provides the error estimate for the approximation in space of the initial value problem \eqref{eq-general-q-FV-sch} by using the Galerkin scheme \eqref{eq-general-qh-FV-sch}.
\begin{proposition}\label{prop-pihq-qh-sch}
Given $q_0 \in \Dzs$ and $q_{0,h}\in X_h$, let $q$ and $q_h$ be the solutions of \eqref{eq-general-q-FV-sch} and \eqref{eq-general-qh-FV-sch} respectively. Assume that $C^*C \in \mathcal{L}\left(\Dz\right)$. Then, there exist $M>0$ and $h^*>0$ such that for all $t\in[0,\tau]$ and all $h\in(0,h^*)$ 
\begin{multline*}
\| \pi_hq(t)-q_h(t) \| \leq \| \pi_hq_0-q_{0,h} \| + Mh^\theta \Big[ t\left( \| q_0 \|_2 + \| F \|_{1,\infty} \right) + t^2 \| F \|_{2,\infty} \Big]\\
+ \int_0^t \| F(s)-F_h(s) \| ds.
\end{multline*}
\end{proposition}
\begin{proof}
First, we substract \eqref{eq-general-qh-FV-sch} from \eqref{eq-general-q-FV-sch} and obtain (we omit the time dependence for the sake of clarity) for all $\varphi_h\in X_h$ 
$$
\left\langle \dot{q}-\dot{q}_h,\varphi_h \right\rangle = \pm i\left\langle q-q_h,\varphi_h \right\rangle_\frac{1}{2} - \left\langle C^*C(q-q_h),\varphi_h \right\rangle + \left\langle F-F_h,\varphi_h \right\rangle.
$$
Noting that $\left\langle \pi_hq-q,\varphi_h \right\rangle_\frac{1}{2} = 0$ for all $\varphi_h\in X_h$ and that $\pi_h\dot{q}$ makes sense by the regularity of $q$ (see \eqref{eq-reg-q}), we obtain from the above equality that for all $\varphi_h\in X_h$ 
\begin{multline}\label{eq-tech-pihq-qh-sch}
\left\langle \pi_h\dot{q}-\dot{q}_h,\varphi_h \right\rangle = \left\langle \pi_h\dot{q}-\dot{q},\varphi_h \right\rangle \pm i\left\langle \pi_hq-q_h,\varphi_h \right\rangle_\frac{1}{2}\\
-\left\langle C^*C\left(q-q_h\right),\varphi_h \right\rangle + \left\langle F-F_h,\varphi_h \right\rangle.
\end{multline}
On the other hand, setting 
$$
\mathcal{E}_h = \frac{1}{2} \|\pi_hq-q_h\|^2,
$$
we have 
$$
\dot{\mathcal{E}}_h = \mbox{Re} \, \left\langle \pi_h\dot{q}-\dot{q}_h,\pi_hq-q_h\right\rangle.
$$
Applying \eqref{eq-tech-pihq-qh-sch} with $\varphi_h=\pi_hq-q_h$ and substituting the result in the above relation, we obtain by using Cauchy-Schwarz inequality and the boundedness of $C$ that there exists $M>0$ such that 
$$
\dot{\mathcal{E}}_h \leq \left( \|\pi_h\dot{q}-\dot{q}\| + M \|\pi_hq-q\| + \|F-F_h\| \right) \underbrace{\|\pi_hq-q_h\|}_{=\sqrt{2\mathcal{E}_h}}.
$$
Since $ \dfrac{\dot{\mathcal{E}}_h}{\sqrt{2\mathcal{E}_h}}=\dfrac{d}{dt}\sqrt{2\mathcal{E}_h}$, the integration of the above inequality from $0$ to $t$ yields 
\begin{multline}\label{eq-intermediate-sch}
\|\pi_hq(t)-q_h(t)\| \leq \|\pi_hq_0-q_{0,h}\| + \int_0^t \left(\|\pi_h\dot{q}(s)-\dot{q}(s)\| + M \|\pi_hq(s)-q(s)\|\right)ds \\
+\int_0^t\|F(s)-F_h(s)\|ds.
\end{multline}
Thus, it remains to bound $\|\pi_h\dot{q}(t)-\dot{q}(t)\|$ and $\|\pi_hq(t)-q(t)\|$ for all $t\in [0,\tau]$. Using \eqref{eq-est-X} and the classical continuous embedding from ${\cal D}(A^\alpha)$ to ${\cal D}(A^\beta)$ for $\alpha>\beta$, we get that 
$$
\left\{
\begin{array}{l}
\|\pi_h\dot{q}(t)-\dot{q}(t)\| \leq M h^\theta \|\dot{q}(t)\|_\frac{1}{2} \leq M h^\theta \|\dot{q}(t)\|_1,\\
\|\pi_hq(t)-q(t)\| \leq M h^\theta \|q(t)\|_\frac{1}{2} \leq M h^\theta \|q(t)\|_2,
\end{array}
\right.
\qquad \forall t \in [0,\tau], \; h \in (0,h^*).
$$
Using relations \eqref{eq-q-norm-1-2} and \eqref{eq-dotq-norm-1} proved in Lemma \ref{lem-annexe2} of the Appendix, we get for all $t \in [0,\tau]$ and all $h \in (0,h^*)$ 
$$
\|\pi_h\dot{q}(t)-\dot{q}(t)\| + \|\pi_hq(t)-q(t)\| \leq M h^\theta \left( \|q_0\|_2 + t \|F\|_{2,\infty} + \|F\|_{1,\infty} \right).
$$
Substituting the above inequality in \eqref{eq-intermediate-sch}, we get the result.
\end{proof}
Using the last result, we derive an error approximation for the semigroups $\T^\pm$ and for the operator $\Lb_t=\T_t^-\T_t^+$.
\begin{proposition}\label{prop-L-tau-h-sch}
Under the assumptions of Proposition \ref{prop-pihq-qh-sch}, the following assertions hold true
\begin{enumerate}
\item There exist $M>0$ and $h^*>0$ such that for all $t\in(0,\tau)$ and all $h\in(0,h^*)$ 
\begin{equation}\label{eq-cor-pihq-qh-forward-sch}
\left\|\pi_h\T^+_tq_0 - \T^+_{h,t} q_0\right\| \leq M t h^\theta \|q_0\|_2.
\end{equation}
\begin{equation}\label{eq-cor-pihq-qh-backward-sch}
\left\|\pi_h\T^-_tq_0 - \T^-_{h,t} q_0\right\| \leq M (\tau-t) h^\theta \|q_0\|_2.
\end{equation}

\item There exist $M>0$ and $h^*>0$ such that for all $n\in\mathbb{N}$, all $t\in[0,\tau]$ and all $h\in(0,h^*)$, we have 
\begin{equation}\label{eq-cor-pihq-qh-sch2}
\| \Lb_t^n q_0 - \Lb_{h,t}^n q_0 \| \leq M(1 + n \tau ) h^\theta \| q_0 \|_2.
\end{equation}
\end{enumerate}
\end{proposition}
\begin{proof}\ 

1. It suffices to take $F=F_h=0$ and $q_{0,h}=\pi_hq_0$ in Proposition \ref{prop-pihq-qh-sch}.

2. We first note that 
\begin{equation}\label{eq-Lt-ineq-tri-sch}
\| \Lb^n_tq_0 - \Lb^n_{h,t} q_0 \| \leq \| \Lb^n_tq_0 - \pi_h\Lb^n_tq_0 \| + \| \pi_h\Lb^n_t q_0 - \Lb^n_{h,t} q_0 \|.
\end{equation}
Using \eqref{eq-est-X} and the fact that $\|\Lb_t\|_{\mathcal{L}({\cal D}(A))} \leq 1$ proved in Lemma \ref{lem-contraction} of the Appendix, the first term in the above relation can be estimated as follows 
\begin{equation}\label{eq-Lt-first-term-sch}
\| \Lb^n_tq_0 - \pi_h\Lb^n_tq_0 \| \leq Mh^\theta\|q_0\|_2, \quad \forall h \in (0,h^*).
\end{equation}
For the second term in \eqref{eq-Lt-ineq-tri-sch}, we prove by induction that for all $n\in\mathbb{N}$
\begin{equation}\label{eq-Lt-second-term-sch}
\| \pi_h\Lb^n_t q_0 - \Lb^n_{h,t} q_0 \| \leq M n\tau h^\theta \| q_0 \|_2, \quad \forall h \in (0,h^*).
\end{equation}
By definition, we have 
$$
\begin{array}{ll}
\| \pi_h\Lb_tq_0 - \Lb_{h,t} q_0 \| &= \| \pi_h\T^-_t\T^+_tq_0 - \T_{h,t}^-\T_{h,t}^+ q_0 \|, \\
&\leq \| \pi_h\T^-_t\T^+_tq_0 - \T_{h,t}^-\T_t^+ q_0 \| + \| \T_{h,t}^-(\T_t^+ q_0 - \T_{h,t}^+ q_0 ) \|.
\end{array}
$$
By Lemma \ref{lem-contraction} of the Appendix and equation \eqref{eq-cor-pihq-qh-backward-sch}, we get 
$$
\| \pi_h\T^-_t\T^+_tq_0 - \T_{h,t}^-\T_t^+ q_0 \| \leq M (\tau-t) h^\theta \|q_0\|_2, \quad \forall h \in (0,h^*).
$$
Obviously $\|\T^-_h\|_{\mathcal{L}(X)}$ is uniformly bounded with respect to $h$ (this follows for example from \eqref{eq-cor-pihq-qh-backward-sch}), and thus by \eqref{eq-est-X} and equation \eqref{eq-cor-pihq-qh-forward-sch}, we have 
$$
\begin{array}{rcl}
\dsp \| \T_{h,t}^-(\T_t^+ q_0 - \T_{h,t}^+ q_0 ) \| &\leq& \dsp \| \T_t^+ q_0 - \pi_h\T_t^+ q_0 \| + \|\pi_h\T_t^+ q_0- \T_{h,t}^+ q_0  \| \\
 & \leq &\dsp  M t h^\theta \|q_0\|_2, \quad \forall h \in (0,h^*).
\end{array}
$$
Consequently
\begin{equation}\label{eq-pih-Lt-sch}
\| \pi_h\Lb_tq_0 - \Lb_{h,t} q_0 \| \leq M \tau h^\theta \|q_0\|_2, \quad \forall h \in (0,h^*),
\end{equation}
which shows that \eqref{eq-Lt-second-term-sch} holds for $n=1$. Suppose now that for a given $n\geq 2$, there holds 
\begin{equation}\label{eq-pih-Lt-rec-sch}
\| \pi_h\Lb^{n-1}_t q_0 - \Lb^{n-1}_{h,t} q_0 \| \leq M (n-1)\tau h^\theta \| q_0 \|_2.
\end{equation}
We write 
$$
\| \pi_h\Lb^n_t q_0 - \Lb^n_{h,t} q_0 \| \leq \| \pi_h \Lb_t \Lb^{n-1}_t q_0 - \Lb_{h,t} \Lb^{n-1}_t q_0 \| + \| \Lb_{h,t}(\Lb^{n-1}_t q_0 - \Lb^{n-1}_{h,t} q_0) \|.
$$
Thanks to Lemma \ref{lem-contraction} and to the uniform boundedness of $\|\Lb_{h,t}\|_{\mathcal{L}(X)}$ with respect to $h$ (which follows from the uniform boundedness of $\|\T_{h,t}^\pm\|$) and using \eqref{eq-pih-Lt-sch} and \eqref{eq-pih-Lt-rec-sch}, we obtain 
$$
\| \pi_h\Lb^n_t q_0 - \Lb^n_{h,t} q_0 \| \leq M ( \tau + (n-1)\tau ) h^\theta \| q_0 \|_2,
$$
which is exactly \eqref{eq-Lt-second-term-sch}. Substituting \eqref{eq-Lt-first-term-sch} and \eqref{eq-Lt-second-term-sch} in \eqref{eq-Lt-ineq-tri-sch}, we obtain the result.
\end{proof}
We are now able to prove Theorem \ref{th-main-sch}. 
\begin{proof}[of Theorem \ref{th-main-sch}]
Introducing the term $\dsp \sum_{n=0}^{N_h} \Lb_{h,\tau}^n z^-(0) $, we rewrite $z_0 - z_{0,h}$ in the following form
$$
\begin{array}{lll}
z_0 - z_{0,h} &= \dsp \sum_{n=0}^\infty \Lb_\tau^n z^-(0) -\sum_{n=0}^{N_h} \Lb_{h,\tau}^n z_{h}^-(0),\\
&= \dsp \sum_{n>N_h} \Lb_\tau^n z^-(0) + \sum_{n=0}^{N_h}\left(\Lb_\tau^n - \Lb_{h,\tau}^n\right) z^-(0)+ \sum_{n=0}^{N_h} \Lb_{h,\tau}^n \left(z^-(0)- z_h^-(0)\right).
\end{array} 
$$
Therefore, we have 
\begin{equation}\label{eq-diff-z0-z0h-sch}
\| z_0 - z_{0,h} \| \leq S_1+S_2+S_3,
\end{equation}
where we have set 
$$
\left\{
\begin{array}{lll}
\dsp S_1= \sum_{n>N_h} \left\| \Lb_\tau^n z^-(0) \right\|,\\
\dsp S_2=\sum_{n=0}^{N_h} \left\| \left(\Lb_\tau^n - \Lb_{h,\tau}^n\right) z^-(0) \right\|,\\
\dsp S_3=\left(\sum_{n=0}^{N_h} \left\| \Lb_{h,\tau}^n \right\|_{\mathcal{L}(X)} \right) \left\| z^-(0) - z_h^-(0) \right\|.
\end{array} 
\right.
$$
Note that the term $S_1$ is the truncation error of the tail of the infinite sum \eqref{eq-z0-neumann}, the term $S_2$ represents the cumulated error due to the approximation of the semigroups $\T^\pm$ while the term $S_3$ comes from the approximation of the first iterate $z^-(0)$ of the algorithm. 

Since $\eta=\|\Lb_\tau\|_{\mathcal{L}(X)} < 1$, using relation \eqref{eq-z0-Ltau}, the first term can be estimated very easily 
\begin{equation}\label{eq-first-term-th-sch}
S_1\leq M\, \frac{\eta^{N_h+1}}{1-\eta} \| z_0 \|_2.
\end{equation}
The term $S_2$ can be estimated using the estimate \eqref{eq-cor-pihq-qh-sch2} from Proposition \ref{prop-L-tau-h-sch} 
$$
S_2 \leq M \left(\sum_{n=0}^{N_h} (1+n \tau)\right) h^\theta \|z^-(0)\|_2, \quad \forall h \in (0,h^*).
$$
Therefore, using \eqref{eq-z0-Ltau} and the fact that $\|\Lb_\tau\|_{\Ds}<1$ (see Lemma \ref{lem-contraction}) in the above relation, we finally get that 
\begin{equation}\label{eq-second-term-th-sch}
S_2 \leq M\Big[1+(1+\tau)N_h+N_h^2\tau \Big] h^\theta \|z_0\|_2, \quad \forall h \in (0,h^*).
\end{equation}
It remains to estimate the term $S_3$. As $\eta=\|\Lb_\tau\|_{\mathcal{L}(X)} < 1$, \eqref{eq-cor-pihq-qh-sch2} implies that $\|\Lb_{h,\tau}\|_{\mathcal{L}(X)}$ is also uniformly with respect to $h$  bounded by 1, provided $h$ is small enough. Hence, we have 
\begin{equation}\label{eq-S3-num1}
\begin{array}{rcl}
S_3 &\leq & \dsp M N_h \left\| z^-(0) - z_h^-(0) \right\|\\
 & \leq &MN_h \dsp\left(\left\| z^-(0) - \pi_h z^-(0)\right\|+ \left\|\pi_h z^-(0)- z_h^-(0)\right\|\right).
\end{array}
\end{equation}
By using \eqref{eq-est-X} and \eqref{eq-z0-Ltau}, we immediately obtain that
\begin{equation}\label{eq-S3-num2}
\left\| z^-(0) - \pi_h z^-(0)\right\| \leq Mh^\theta \|z_0\|_2.
\end{equation}
To estimate the second term $\pi_h z^-(0)- z_h^-(0)$, we apply twice Proposition \ref{prop-pihq-qh-sch} first for the time reversed backward observer $z^-(\tau-\cdot)$ and then for the forward observer $z^+$ (the time reversal step is introduced  as in the formulation of Proposition \ref{prop-pihq-qh-sch}, only initial value Cauchy problems can be considered). After straightforward calculation we obtain that for all $h\in(0,h^*)$ 
\begin{multline}\label{eq-intermediate-main-1-sch}
\left\| \pi_h z^-(0) - z_h^-(0) \right\| \leq M h^\theta\Big[ \tau(\|z^+(\tau)\|_2+ \|C^*y\|_{1, \infty}) + \tau^2\|C^*y\|_{2,\infty} \Big]\\
+ \int_0^\tau\|C^*\left(y(\tau-s)-y_h(\tau-s)\right)\|ds + \int_0^\tau\|C^*\left(y(s)-y_h(s)\right)\|ds.
\end{multline}
Applying \eqref{eq-q-norm-1-2} of Lemma \ref{lem-annexe2} of the Appendix with zero initial data, we obtain that 
$$
\|z^+(\tau)\|_2 \leq \tau \|C^*y\|_{2,\infty}.
$$
Therefore \eqref{eq-intermediate-main-1-sch} also reads
$$
\left\| \pi_h z^-(0) - z_h^-(0) \right\| \leq M h^\theta(\tau + \tau^2)\|C^*y\|_{2,\infty} + 2\int_0^\tau\|C^*\left(y(s)-y_h(s)\right)\|ds.
$$
As $C^*C \in \mathcal{L}\left(\Dzs\right) \cap \mathcal{L}\left(\Dz\right)$ and $\|z\|_{2,\infty}=\|z_0\|_2$ (since $iA_0$ is skew-adjoint), the last relation becomes 
$$
\left\| \pi_h z^-(0) - z_h^-(0) \right\| \leq M h^\theta(\tau + \tau^2)\|z_0\|_2 + 2\int_0^\tau\|C^*\left(y(s)-y_h(s)\right)\|ds.
$$
Substituting the above relation and \eqref{eq-S3-num2} in \eqref{eq-S3-num1}, we get
\begin{equation}\label{eq-third-term-th-sch}
S_3 \leq M N_h \left(h^\theta (1+\tau + \tau^2)\|z_0\|_2 + \int_0^\tau\|C^*\left(y(s)-y_h(s)\right)\|ds \right).
\end{equation}
Substituting \eqref{eq-first-term-th-sch}, \eqref{eq-second-term-th-sch} and \eqref{eq-third-term-th-sch} in \eqref{eq-diff-z0-z0h-sch}, we get for all $h\in(0,h^*)$ 
\begin{multline*}
\| z_0 - z_{0,h} \| \leq M \Bigg[ \left(\frac{\eta^{N_h+1}}{1-\eta} +h^\theta \left[1+(1+\tau+\tau^2)N_h+\tau N_h^2 \right] \right)\| z_0 \|_2 \\
+ N_h \int_0^\tau \| C^*\left(y(s)-y_h(s)\right) \|ds \Bigg],
\end{multline*}
which leads to the result (with possibly reducing the value of $h^*$).
\end{proof}

\subsection{Full Discretization}
\label{Subsct_Schr2}
\subsubsection{Statement of the main result}
In order to approximate \eqref{eq-general-q-FV-sch}, we use a finite difference scheme in time combined with the previous Galerkin approximation in space. In others words, we discretize the time interval $[0,\tau]$ using a time step $\Delta t>0$. We obtain a discretization $t_k = k\Delta t$, where $0\le k \le K$ and where we assumed, without loss of generality, that $\tau =K\Delta t$. Given a continuously differentiable function of time $f$, we approximate its derivative at time $t_k$ by the formula
$$
f'(t_k) \simeq D_tf(t_k) := \frac{f(t_k)-f(t_{k-1})}{\Delta t}.
$$
We suppose that $q_{0,h}\in X_h$ and $F^k_{h}$, for $0\le k \le K$, are given approximations of $q_{0}$ and $F(t_k)$ in the space $X$. We define $(q_{h}^k)$, for $0\le k\le K$, as the solution of the following problem: for all $\varphi_h \in X_h$:
\begin{equation}\label{eq-general-qhk-FVDF-sch}
\left\{\begin{array}{ll}
\left\langle D_tq_{h}^k,\varphi_h\right\rangle = \pm i\left\langle q_{h}^k,\varphi_h \right\rangle_\frac{1}{2} - \left\langle C^*Cq_{h}^k,\varphi_h \right\rangle + \left\langle F_{h}^k,\varphi_h \right\rangle,\\
q_{h}^0=q_{0,h}.
\end{array}\right.
\end{equation}
Note that the above procedure leads to a natural approximation $\T^\pm_{h,\Delta t,k}$ of the continuous semigroup $\T^\pm_{t_k}$ by setting
$$
\T^+_{t_k} q_0 \simeq \T^+_{h,\Delta t,k}q_0 := q^k_{h},\qquad\qquad
\T^-_{t_k} q_0 \simeq \T^-_{h,\Delta t,k}q_0 := q^{K-k}_{h},
$$
where $q^k_{h}$ solves \eqref{eq-general-qhk-FVDF-sch} with $F^k_{h}= 0$ for all $0\le k \le K$  and for $q_{0,h}=\pi_h q_0$. Obviously, this also leads to an approximation of $\Lb_\tau=\T^-_\tau\T^+_\tau$ by setting
$$
\Lb_{h,\Delta t,K} = \T^-_{h,\Delta t,K}\T^+_{h,\Delta t,K}.
$$

Assume that for all $0\le k \le K$, $y_{h}^k$ is a given approximation of $y(t_k)$ in $Y$ and  let $\left(z_{h}^+\right)^k$ and $\left(z_{h}^-\right)^k$ be respectively  the approximations of \eqref{eq-sys-obs-forward-schrodinger} and \eqref{eq-sys-obs-backward-schrodinger}  obtained via \eqref{eq-general-qhk-FVDF-sch} as follows:
\begin{itemize}
\item For all $0\le k \le K$, $\left(z_{h}^+\right)^k=q^k_{h}$ where $q^k_{h}$ solves \eqref{eq-general-qhk-FVDF-sch} with  $F^k_{h} = C^*y_{h}^k$ and $q_{h}^0=0$,
\item For all $0\le k \le K$, $\left(z_{h}^-\right)^k=q^{K-k}_{h}$ where $q^k_{h}$ solves \eqref{eq-general-qhk-FVDF-sch} with $F^k_{h} = C^*y^{K-k}_h $ and $q_{h}^0=(z_h^+)^K$.
\end{itemize}
Then, our main result (which is the fully discrete counterpart of Theorem \ref{th-main-sch}) reads as follows 
\begin{theorem}\label{th-main-full-sch}
Let $A_0 : \Dz \rightarrow X$ be a strictly positive self-adjoint operator and $C \in \mathcal{L}(X,Y)$ such that $C^*C \in \mathcal{L}\left(\Dzs\right)\cap\mathcal{L}\left(\Dz\right)$. We assume that the pair $(iA_0,C)$ is exactly observable in time $\tau>0$. Let $z_0\in \Dzs$ be the initial value of \eqref{eq-sys-initial-schrodinger}. With the above notation, let $z_{0,h,\Delta t}$ be defined by \eqref{eq-z0hdelta-neumann} and denote $\eta:=\|\Lb_\tau\|_{\mathcal{L}(X)} < 1$. Then there exist $M > 0$, $h^* > 0$ and $\Delta t^*>0$ such that for all $h \in (0, h^*)$ and all $\Delta t\in(0,\Delta t^*)$ we have
\begin{multline*}
\| z_0 - z_{0,h,\Delta t} \| \leq M \Bigg[ \left(\frac{\eta^{N_{h,\Delta t}+1}}{1-\eta} +  (h^\theta+\Delta t)(1+\tau) N_{h,\Delta t}^2 \right)\| z_0 \|_2\\
+ N_{h,\Delta t}\Delta t \sum_{\ell=0}^{K} \big\|C^*(y(t_\ell)-y_h^\ell)\big\| \Bigg].
\end{multline*}
\end{theorem}
\begin{corollary}
Under the assumptions of Theorem \ref{th-main-full-sch}, we set 
$$
N_{h,\Delta t} = \frac{\ln (h^\theta + \Delta t)}{\ln \eta}
$$
Then, there exist $ M_\tau > 0 $, $ h^* > 0 $ and $\Delta t^*>0$ such that for all $h \in (0, h^*)$ and $\Delta t\in(0,\Delta t^*)$ 
\begin{multline}\label{eqz0z0hk}
\| z_0 - z_{0,h,\Delta t} \| \leq M_\tau \bigg[(h^\theta+\Delta t) \ln^2 (h^\theta+\Delta t) \| z_0 \|_2\\
+ \left|\ln (h^\theta + \Delta t)\right| \Delta t \sum_{\ell=0}^{K}\big\|C^*(y(t_\ell)-y_h^\ell)\big\|  \bigg].
\end{multline}
\end{corollary}
\begin{remark}
Contrarily to the semi-discrete case, we have not been able to extend our results for $z_0$ in a larger space than $\Dzs$. 
\end{remark}
\subsubsection{Proof of Theorem \ref{th-main-full-sch}}
The proof of Theorem \ref{th-main-full-sch} goes along the same lines as the one of Theorem \ref{th-main-sch} in the semi-discrete case and uses energy estimates similar to those developed in Fujita and Suzuki \cite[p.\, 865]{FujSuz91}. The main ingredient for the convergence analysis is the following result (the counterpart of Proposition \ref{prop-pihq-qh-sch}) which gives the error estimate for the approximation (in space and time) of system \eqref{eq-general-q-FV-sch} by \eqref{eq-general-qhk-FVDF-sch}.
\begin{proposition}\label{prop-pihq-qhk-sch}
Given initial states $q_0\in\Dzs$ and $q_{0,h} \in X_h$, let $q$ and $q_{h}^k$, for $0\le k \le K$, be respectively the solutions of \eqref{eq-general-q-FV-sch} and \eqref{eq-general-qhk-FVDF-sch}. Assume that $C^*C \in \mathcal{L}\left(\Dz\right)$. Then, there exist $M>0$, $h^*>0$ and $\Delta t^*>0$ such that for all $h\in(0,h^*)$, all $\Delta t\in(0,\Delta t^*)$ and all  $0\le k \le K$:
\begin{multline*}
\| \pi_hq(t_k)-q_{h}^k \| \leq \| \pi_hq_0 - q_{0,h} \| + M \bigg\{\Delta t \sum_{\ell=1}^{k}\|F(t_\ell)-F_{h}^\ell\|\\
+\left(h^\theta + \Delta t \right)\Big[ t_k\big(\|q_0\|_2 + \|F\|_{1, \infty} 
+ \|\dot{F}\|_\infty\big) + t_k^2\|F\|_{2,\infty} \Big] \bigg\}.
\end{multline*}
\end{proposition}
\begin{proof}
Let $r_1(t_k)$ denote the residual term in the first order Taylor expansion of $q$ around $t_{k-1}$, so that
\begin{equation}\label{eq-taylor-sch}
\dot{q}(t_k) = \frac{q(t_k)-q(t_{k-1})}{\Delta t} - \frac{1}{\Delta t}r_1(t_k) = D_tq(t_k) - \frac{1}{\Delta t}r_1(t_k),
\end{equation}
Subtracting \eqref{eq-general-qhk-FVDF-sch} from the continuous weak formulation \eqref{eq-general-q-FV-sch} applied for $t=t_k$ and for an arbitrary test function $\varphi=\varphi_h\in X_h$, we immediately get by using \eqref{eq-taylor-sch} that for all $1\le k \le K$ 
\begin{multline*}
\left\langle D_t\left(q(t_k)-q_{h}^k\right),\varphi_h\right\rangle = 
\pm i\left\langle \pi_h q(t_k),\varphi_h\right\rangle_\frac{1}{2} - \left\langle C^*C\left(q(t_k)-q_{h}^k\right),\varphi_h\right\rangle \\
+ \frac{1}{\Delta t} \left\langle r_1(t_k),\varphi_h\right\rangle + \left\langle F(t_k)-F_{h}^k,\varphi_h\right\rangle.
\end{multline*}
The above relation implies that
\begin{multline}\label{eq-tech-pihq-qhk-sch}
\left\langle D_t\left(\pi_h q(t_k)-q_{h}^k\right),\varphi_h\right\rangle = \left\langle D_t\left(\pi_h q(t_k)-q(t_k)\right),\varphi_h\right\rangle \\
\pm i\left\langle \pi_h q(t_k)-q_{h}^k,\varphi_h\right\rangle_\frac{1}{2} - \left\langle C^*C\left(q(t_k)-q_{h}^k\right),\varphi_h\right\rangle \\
+ \frac{1}{\Delta t} \left\langle r_1(t_k),\varphi_h\right\rangle + \left\langle F(t_k)-F_{h}^k,\varphi_h\right\rangle .
\end{multline}
Now, for all $1\le k \le K$, let
$$
\mathcal{E}_{h}^k = \frac{1}{2} \|\pi_h q(t_k)-q_{h}^k\|^2.
$$
Using the identity 
$$
\frac{1}{2} \left( \|u\|^2 - \|v\|^2 + \|u-v\|^2 \right) = \mbox{Re} \, \left\langle u-v,u \right\rangle, \quad \forall u,v \in X,
$$
one easily obtains that for all $1\le k \le K$
$$
D_t\mathcal{E}_{h}^k \leq \mbox{Re} \, \left\langle D_t\left(\pi_h q(t_k)-q_{h}^k\right),\pi_h q(t_k)-q_{h}^k\right\rangle.
$$
Substituting \eqref{eq-tech-pihq-qhk-sch} with $\varphi_h = \pi_h q(t_k)-q_{h}^k$ in the above inequality and using the boundedness of $C$, we obtain the existence of $M>0$ such that for all $1\le k \le K$
\begin{multline}\label{eq-intermediate-full-1-sch}
D_t\mathcal{E}_{h}^k \leq \big[\|D_t\left(\pi_hq(t_k)-q(t_k)\right)\| + M \|\pi_hq(t_k)-q(t_k)\| \\
+\frac{1}{\Delta t} \|r_1(t_k)\| + \|F(t_k)-F_{h}^k\|\big]\|\pi_hq(t_k)-q_{h}^k\|.
\end{multline}
Using the straightforward relations
\begin{equation}\label{eq-DtId-1}
D_t\mathcal{E}_{h}^k = \left(D_t\sqrt{\mathcal{E}_{h}^k}\right)\, \left(\sqrt{\mathcal{E}_{h}^k}+\sqrt{\mathcal{E}_{h}^{k-1}}\right),
\end{equation}
and 
\begin{equation}\label{eq-DtId-2}
\|\pi_h q(t_k)-q_{h}^k\| \leq \sqrt{2}\left(\sqrt{\mathcal{E}_{h}^k}+\sqrt{\mathcal{E}_{h}^{k-1}}\right),
\end{equation}
we obtain from \eqref{eq-est-X} and \eqref{eq-intermediate-full-1-sch} that for all $h\in(0,h^*)$ 
$$
D_t\sqrt{\mathcal{E}_{h}^k} \leq M\left\{h^\theta \left( \|D_tq(t_k)\|_\frac{1}{2} + \|q(t_k)\|_\frac{1}{2} \right) + \frac{1}{\Delta t} \|r_1(t_k)\| + \|F(t_k)-F_{h}^k\|\right\}.
$$
By \eqref{eq-taylor-sch} and relations \eqref{eq-q-norm-1-2} and \eqref{eq-dotq-norm-1} in Lemma \ref{lem-annexe2} of the Appendix, the last estimate yields
\begin{multline}\label{eq-intermediate-full-2-sch}
\hspace{2cm}D_t\sqrt{\mathcal{E}_{h}^k} \leq M\bigg\{h^\theta \left( \|q_0\|_2 + t_k\|F\|_{2,\infty} + \|F\|_{1,\infty} \right)  \\
+ \|F(t_k)-F_{h}^k\|+ \frac{h^\theta}{\Delta t}\|r_1(t_k)\|_\frac{1}{2} + \frac{1}{\Delta t}\|r_1(t_k)\|\bigg\}.
\end{multline}
To conclude, it remains to bound the two last terms in the above estimate. By definition of $r_1$, we have
$$
r_1(t_k) = q(t_{k-1})-q(t_k) + \Delta t \,\dot{q}(t_k),
$$
in $\Dzdemi$, and thus by the mean value theorem, we get
$$
\|r_1(t_k)\|_\frac{1}{2} \leq \Delta t \sup_{s\in\left[t_{k-1},t_k\right]}\|\dot{q}(s)\|_\frac{1}{2} + \Delta t\|\dot{q}(t_k)\|_\frac{1}{2}.
$$
Using once again \eqref{eq-dotq-norm-1}, we obtain that there exists $M>0$ such that
\begin{equation}\label{eq-intermediate-full-3-sch}
\|r_1(t_k)\|_\frac{1}{2} \leq M \Delta t \left( \|q_0\|_2 + t_k\|F\|_{2,\infty} + \|F\|_{1,\infty} \right).
\end{equation}
Now by the regularity of $q$ (see Lemma \ref{lem-annexe2}), the residual $r_1$ can be expressed via the integral 
$$
r_1(t_k) = \int_{t_{k-1}}^{t_k} \ddot{q}(s)\left(t_{k-1}-s\right)ds,
$$
in $X$, and thus 
$$
\|r_1(t_k)\| \leq \Delta t^2 \sup_{s\in\left[t_{k-1},t_k\right]}\|\ddot{q}(s)\|.
$$
Using equation \eqref{eq-general-q-sch} verified by $q$ and the boundedness of $C$, we have
$$
\begin{array}{lll}
\|\ddot{q}(t)\| = \Big\|\dfrac{d\dot{q}}{dt}(t) \Big\|&= \Big\|\dfrac{d}{dt}\Big\{\pm i A_0 q(t) -C^*Cq(t) + F(t) \Big\} \Big\|, \\
&\leq \|\dot{q}(t)\|_1 + M \|\dot{q}(t)\| + \|\dot{F}(t)\|.
\end{array}
$$
Hence, once again by \eqref{eq-dotq-norm-1}, we get 
\begin{equation}\label{eq-intermediate-full-4-sch}
\|r_1(t_k)\| \leq \Delta t^2 \left( \|q_0\|_2 + t_k\|F\|_{2,\infty} + \|F\|_{1,\infty} + \|\dot{F}\|_\infty\right).
\end{equation}
Substituting inequalities \eqref{eq-intermediate-full-3-sch} and \eqref{eq-intermediate-full-4-sch} in relation \eqref{eq-intermediate-full-2-sch} provides estimates for $D_t\sqrt{\mathcal{E}_{h}^k}=\dfrac{\sqrt{\mathcal{E}_{h}^k}-\sqrt{\mathcal{E}_{h}^{k-1}}}{\Delta t}$, for $k=1,\dots,K$, that can be added together to get  the desired inequality (since $\|\pi_hq(t_k)-q^k_h\|=\sqrt{2\mathcal{E}_{h}^k} $).
\end{proof}
Using this Proposition, we can derive an error estimate for the semigroup $\T^\pm_{t_k}$ (for all $1\le k\le K$) and for the operator $\Lb_\tau = \T^-_\tau\T^+_\tau$ (the counterpart of Proposition \ref{prop-L-tau-h-sch}).
\begin{proposition}\label{prop-L-tau-hk-sch}
Under the assumptions of Proposition \ref{prop-pihq-qhk-sch}, the following assertions hold true
\begin{enumerate}
\item There exist $M>0$, $h^*>0$ and $\Delta t^*>0$ such that for all $h\in(0,h^*)$, all $\Delta t\in(0,\Delta t^*)$ and all $0\le k\le K$
\begin{equation}\label{eq-cor-pihq-qhk-forward-sch}
\left\|\pi_h\T^+_{t_k} q_0 - \T^+_{h,\Delta t, k} q_0\right\| \leq M t_k(h^\theta + \Delta t) \|q_0\|_2.
\end{equation}
\begin{equation}\label{eq-cor-pihq-qhk-backward-sch}
\left\|\pi_h\T^-_{t_k} q_0 - \T^-_{h,\Delta t, k} q_0\right\| \leq M (\tau-t_k)(h^\theta + \Delta t) \|q_0\|_2.
\end{equation}

\item There exist $M>0$, $h^*>0$ and $\Delta t^*>0$ such that for all $n\in\mathbb{N}$, all $h\in(0,h^*)$, all $\Delta t\in(0,\Delta t^*)$ and all $0\le k\le K$
\begin{equation}\label{eq-cor-pihq-qhk-sch2}
\| (\Lb_{t_k}^n - \Lb_{h,\Delta t,k}^n) q_0 \| \leq M \left[ h^\theta + n \tau\left(h^\theta+\Delta t\right) \right] \| q_0 \|_2.
\end{equation}

\end{enumerate}
\end{proposition}
\begin{proof}\ 

1. It suffices to apply Proposition \ref{prop-pihq-qhk-sch} with $F(t_k)=F_{h}^k=0$ for all $0\le k\le K$ and $q_{0,h,\Delta t}=\pi_hq_0$.

2. First, we note that 
\begin{equation}\label{eq-Lt-ineq-tri-full-sch}
\| \Lb^n_{t_k}q_0 - \Lb^n_{h,\Delta t,k} q_0 \| \leq \| \Lb^n_{t_k}q_0 - \pi_h\Lb^n_{t_k}q_0 \| + \| \pi_h\Lb^n_{t_k} q_0 - \Lb^n_{h,\Delta t, k} q_0 \|.
\end{equation}
Using \eqref{eq-est-X}, the fact that $\|\Lb^n_t\|_{\mathcal{L}({\cal D}(A))} \leq 1$ (proved in Lemma \ref{lem-contraction} of the Appendix), the first term in the above relation can be estimated as follows 
\begin{equation}\label{eq-Lt-first-term-full-sch}
\| \Lb^n_{t_k}q_0 - \pi_h\Lb^n_{t_k}q_0 \| \leq Mh^\theta\|q_0\|_2, \quad \forall h \in (0,h^*).
\end{equation}
For the second term in \eqref{eq-Lt-ineq-tri-full-sch}, we prove by induction that for all $n\in\mathbb{N}$, all $h \in (0,h^*)$ and all $\Delta t\in(0,\Delta t^*)$ (for some $\Delta t^*>0$)
\begin{equation}\label{eq-Lt-second-term-full-sch}
\| \pi_h\Lb^n_{t_k}q_0 - \Lb^n_{h,\Delta t, k} q_0 \| \leq M n \tau \left(h^\theta + \Delta t \right) \| q_0 \|_2.
\end{equation}
By definition, we have 
$$
\begin{array}{lll}
\left\| \pi_h \Lb_{t_k}q_0 - \Lb_{h,\Delta t, k} q_0 \right\| &= \big\| \pi_h \T^-_{t_k} \T^+_{t_k}q_0 - \T_{h,\Delta t, k}^- \T_{h,\Delta t, k}^+ q_0\big\|, \\
&\leq \left\| \left(\pi_h \T^-_{t_k} - \T_{h,\Delta t, k}^-\right) \pi_h \T_{t_k}^+ q_0 \right\|\\
&\hspace{1cm} + \left\| \T_{h,\Delta t, k}^-\left( \pi_h \T_{t_k}^+ - \T_{h,\Delta t, k}^+ \right) q_0 \right\|.
\end{array}
$$
Using \eqref{eq-cor-pihq-qhk-backward-sch} and Lemma \ref{lem-contraction}, we get 
$$
\left\| \left( \pi_h \T^-_{t_k} - \T_{h,\Delta t, k}^- \right) \pi_h \T_{t_k}^+ q_0 \right\| \leq M(\tau- t_k) \left( h^\theta + \Delta t \right) \| q_0 \|_2.
$$
Obviously $\| \T^-_{h,\Delta t,k} \|_{\mathcal{L}(X)}$ is uniformly bounded (with respect to $h$ and $\Delta t$), and thus again by \eqref{eq-cor-pihq-qhk-forward-sch} we have
$$
\left\| \T_{h,\Delta t,k}^- \left( \pi_h \T_{t_k}^+ - \T_{h,\Delta t,k}^+ \right) q_0 \right\| \leq M t_k\left( h^\theta + \Delta t \right) \|q_0\|_2.
$$
So, by adding the two last inequalities, we obtain that
\begin{equation}\label{eq-pih-Lt-full-sch}
\left\| \pi_h\Lb_{t_k}q_0 - \Lb_{h,\Delta t,k} q_0 \right\| \leq M \tau\left( h^\theta + \Delta t\right) \|q_0\|_2,
\end{equation}
showing that \eqref{eq-Lt-second-term-full-sch} holds for $n=1$. Suppose now that for some $n\geq 2$ 
\begin{equation}\label{eq-pih-Lt-rec-full-sch}
\| \pi_h\Lb^{n-1}_{t_k} q_0 - \Lb^{n-1}_{h,\Delta t,k} q_0 \| \leq M (n-1)
 \tau \left( h^\theta + \Delta t\right) \| q_0 \|_2.
\end{equation}
Writing 
\begin{multline*}
\| \pi_h\Lb^n_{t_k} q_0 - \Lb^n_{h,\Delta t,k} q_0 \| \leq \| \pi_h\Lb_{t_k}\Lb^{n-1}_{t_k} q_0 - \Lb_{h,\Delta t,k}\pi_h\Lb^{n-1}_{t_k} q_0 \|\\
+ \| \Lb_{h,\Delta t,k}(\pi_h\Lb^{n-1}_{t_k} q_0 - \Lb^{n-1}_{h,\Delta t,k} q_0) \|,
\end{multline*}
we get by using Lemma \ref{lem-contraction}, the uniform boundedness of $\|\Lb_{h,\Delta t,k}\|_{\mathcal{L}(X)}$ with respect to $h$ and $\Delta t$, \eqref{eq-pih-Lt-full-sch} and \eqref{eq-pih-Lt-rec-full-sch} that
$$
\| \pi_h\Lb^n_{t_k} q_0 - \Lb^n_{h,\Delta t,k} q_0 \| \leq M \left[(1+(n-1)) \tau \left( h^\theta + \Delta t\right) \right]\| q_0 \|_2,
$$
which is exactly \eqref{eq-Lt-second-term-full-sch}. Substituting \eqref{eq-Lt-first-term-full-sch} and \eqref{eq-Lt-second-term-full-sch} in \eqref{eq-Lt-ineq-tri-full-sch}, we obtain the result.
\end{proof}
We are now able to prove Theorem \ref{th-main-full-sch}.
\begin{proof}[of Theorem \ref{th-main-full-sch}]
We first introduce the term $\dsp \sum_{n=0}^{N_{h,\Delta t}} \Lb_{h,\Delta t,K}^n z^-(0)$ to rewrite the approximation error $z_0 - z_{0,h,\Delta t}$ in the following form:
$$
\begin{array}{lll}
z_0 - z_{0,h,\Delta t} &= \dsp \sum_{n=0}^\infty \Lb_\tau^n z^-(0) -\sum_{n=0}^{N_{h,\Delta t}} \Lb_{h,\Delta t,K}^n \left(z_{h}^-\right)^0\\
&= \dsp \sum_{n>N_{h,\Delta t}} \Lb_\tau^n z^-(0) + \sum_{n=0}^{N_{h,\Delta t}}\left(\Lb_\tau^n - \Lb_{h,\Delta t,K}^n \right) z^-(0)\\
& \dsp \hspace{5cm} + \sum_{n=0}^{N_{h,\Delta t}} \Lb_{h,\Delta t,K}^n \left( z^-(0) - \left(z_{h}^-\right)^0\right).
\end{array} 
$$
Therefore, we have 
\begin{equation}\label{eq-diff-z0-z0hk-sch}
\| z_0 - z_{0,h,\Delta t} \| \leq S_1+S_2+S_3,
\end{equation}
where we have set 
$$
\left\{
\begin{array}{lll}
\dsp S_1= \sum_{n>N_{h,\Delta t}} \left\| \Lb_\tau^n z^-(0) \right\|,\\
\dsp S_2= \sum_{n=0}^{N_{h,\Delta t}} \left\| \left(\Lb_\tau^n - \Lb_{h,\Delta t,K}^n \right) z^-(0) \right\|,\\
\dsp S_3= \left(\sum_{n=0}^{N_{h,\Delta t}} \left\| \Lb_{h,\Delta t,K}^n \right\|_{\mathcal{L}(X)} \right) \Big\| z^-(0) - \left(z_{h}^-\right)^0 \Big\|.
\end{array} 
\right.
$$
Since $\eta =\|\Lb_\tau\|_{\mathcal{L}(X)} < 1$, the first term can be estimated very easily 
\begin{equation}\label{eq-first-term-th-full-sch}
S_1\leq M \frac{\eta^{N_{h,\Delta t}+1}}{1-\eta} \| z_0 \|_2.
\end{equation}
The second term $S_2$ can be estimated using the estimate \eqref{eq-cor-pihq-qhk-sch2} from Proposition \ref{prop-L-tau-hk-sch}
$$
S_2 \leq M \Bigg\{\sum_{n=0}^{N_{h,\Delta t}} \left(h^\theta + n \tau ( h^\theta + \Delta t)\right)\Bigg\} \|z^-(0)\|_2, \quad \forall h \in (0,h^*), \Delta t \in (0,\Delta t^*).
$$
Therefore, using \eqref{eq-z0-Ltau}, the fact that $\|\Lb_\tau\|_{\Ds}<1$ (see Lemma \ref{lem-contraction}) in the above relation, we get that for all $h \in (0,h^*)$ and $\Delta t\in(0,\Delta t^*)$ 
\begin{equation}\label{eq-second-term-th-full-sch}
S_2 \leq M\Big[1+(1+\tau)N_{h,\Delta t}+(1+\tau)N_{h,\Delta t}^2 \Big] \left(h^\theta + \Delta t\right) \|z_0\|_2.
\end{equation}
It remains to estimate the term $S_3$. As for the semi-discrete case, on can easily show that $\|\Lb_{h,\Delta t,K}\|_{\mathcal{L}(X)}$ is uniformly bounded by 1 (with respect to $h$ and $\Delta t$), and thus we have 
\begin{equation}\label{eq-S3-full-num1}
\begin{array}{rcl}
S_3 &\leq & \dsp M N_{h,\Delta t} \left\| z^-(0) - (z_h^-)^0 \right\|\\
 & \leq & M N_{h,\Delta t} \dsp\left(\left\| z^-(0) - \pi_h z^-(0)\right\|+ \left\|\pi_h z^-(0)- (z_{h}^-)^0\right\|\right).
\end{array}
\end{equation}
By using \eqref{eq-est-X} and \eqref{eq-z0-Ltau}, we immediately obtain that
\begin{equation}\label{eq-S3-full-num2}
\left\| z^-(0) - \pi_h z^-(0)\right\| \leq Mh^\theta \|z_0\|_2.
\end{equation}
To estimate the second term $\pi_h z^-(0)- (z_{h}^-)^0$, we apply twice Proposition \ref{prop-pihq-qhk-sch} first for the time reversed backward observer $z^-(\tau-\cdot)$ and then for the forward observer $z^+$ (the time reversal step is introduced simply because Proposition \ref{prop-pihq-qhk-sch} is written for initial (and not final) value Cauchy problems). After straightforward calculation we obtain that for all $h\in(0,h^*)$ and all $\Delta t\in(0,\Delta t^*)$
\begin{multline}\label{eq-intermediate-main-1-full-sch}
\left\| \pi_h z^-(0) - (z_{h}^-)^0 \right\| \leq M (h^\theta + \Delta t) \Big[ \tau(\|z^+(\tau)\|_2+ \|C^*y\|_{1, \infty} + \|C^*\dot{y}\|_\infty) \\
 + \tau^2\|C^*y\|_{2,\infty} \Big] + \Delta t \sum_{\ell=1}^K \|C^*\left(y(\tau-t_\ell)-y_{h}^{K-\ell}\right)\| + \Delta t \sum_{\ell=1}^K \|C^*\left(y(t_\ell)-y_{h}^\ell\right)\|.
\end{multline}
Applying \eqref{eq-q-norm-1-2} of Lemma \ref{lem-annexe2} of the Appendix with zero initial data, we obtain that 
$$
\|z^+(\tau)\|_2 \leq \tau \|C^*y\|_{2,\infty}.
$$
As $C^*C \in \mathcal{L}\left(\Dzs\right) \cap \mathcal{L}\left(\Dz\right)$ and $\|z\|_{2,\infty}=\|z_0\|_2$ (since $iA_0$ is skew-adjoint), \eqref{eq-intermediate-main-1-full-sch} also reads
$$
\left\| \pi_h z^-(0) - (z_{h}^-)^0 \right\| \leq M (h^\theta + \Delta t) (\tau + \tau^2)\|z_0\|_2 + 2\Delta t \sum_{\ell=0}^K \|C^*\left(y(t_\ell)-y_{h}^\ell\right)\|.
$$
Substituting the above relation and \eqref{eq-S3-full-num2} in \eqref{eq-S3-full-num1}, we get
\begin{equation}\label{eq-third-term-th-full-sch}
S_3 \leq M N_{h,\Delta t} \left\{(h^\theta + \Delta t) (1+\tau + \tau^2)\|z_0\|_2 + \Delta t \sum_{\ell=0}^K \|C^*\left(y(t_\ell)-y_{h}^\ell\right)\| \right\}.
\end{equation}
Substituting \eqref{eq-first-term-th-full-sch}, \eqref{eq-second-term-th-full-sch} and \eqref{eq-third-term-th-full-sch} in \eqref{eq-diff-z0-z0hk-sch}, we get for all $h\in(0,h^*)$ and all $\Delta t\in(0,\Delta t^*)$
\begin{multline*}
\| z_0 - z_{0,h,\Delta t} \| \leq M 
\left\{ N_{h,\Delta t} \Delta t \sum_{\ell=0}^K \left\|C^*\left(y(t_\ell)-y_{h}^\ell\right) \right\|
+ \frac{\eta^{N_{h,\Delta t}+1}}{1-\eta} \| z_0 \|_2 \right.\\
\left.+   (h^\theta + \Delta t) \Big[1+(1+\tau+\tau^2)N_{h,\Delta t} +(1+\tau) N_{h,\Delta t}^2 \Big] \| z_0 \|_2 \right\},
\end{multline*}
which leads to the result (with possibly reducing the value of $h^*$ and $\Delta t^*$).
\end{proof}

\section{The wave equation}
\setcounter{equation}{0}
\label{Sect_Wave}
Let $H$ be a Hilbert space endowed with the inner product $\left\langle\cdot,\cdot\right\rangle$. The corresponding norm of $H$ is denoted by $\|\cdot\|$. Let $A_0 : \Dz \rightarrow H$ be a strictly positive self-adjoint operator and $C_0 \in \mathcal{L}(H,Y)$ a bounded observation operator, where $Y$ is an other Hilbert space. The norm in ${\cal D}(A_0^\alpha)$ will be denoted by $\|\cdot\|_\alpha$. Given $\tau>0$, we deal with the general wave type system 
\begin{equation}\label{eq-sys-initial-wave}
\left\lbrace
\begin{array}{ll}
\ddot{w}(t) + A_0 w(t) = 0, \quad \forall t \geqslant 0,\\
y(t) = C_0 \dot w(t), \quad \forall t \in [0,\tau],
\end{array}\right.
\end{equation}
and we want to reconstruct the initial value $(w_0,w_1)=(w(0),\dot w(0))$ of \eqref{eq-sys-initial-wave} knowing $y(t)$ for $t\in [0,\tau]$.
In order to use the general iterative algorithm described in the introduction, we first rewrite \eqref{eq-sys-initial-wave} as a first order system of the form \eqref{eq-sys-initial}. To achieve this, it suffices to introduce the following notation:
\begin{eqnarray}
z(t) = \left[ \begin{matrix} w(t) \\ \dot{w}(t) \end{matrix} \right],& \quad X = \Dzdemi \times H, \nonumber \\
A = \left( \begin{matrix} 0 & I \\ -A_0 & 0 \end{matrix} \right),& \quad \D = \Dz \times \Dzdemi, \label{eq-def-A} \\
C \in \mathcal{L}(X,Y) ,& \quad C = \left[ \begin{matrix} 0 & C_0 \end{matrix} \right]. \label{eq-def-C}
\end{eqnarray}
The space $X$ is endowed with the norm
$$
\left\| z\right\| =\sqrt{\| z_1\|_{\frac{1}{2}}^2+\|z_2\|^2}, \FORALL z=\left[ \begin{matrix} z_1 \\ z_2 \end{matrix} \right] \in X.
$$
Note that the operator $iA$ is selfadjoint but has no sign so that the problem studied here does not fit into the framework of Section \ref{Sect_Schr}. We  assume that the pair $(A,C)$ is exactly observable in time $\tau>0$. Thus, according to Liu \cite[Theorem 2.3.]{Liu97}, $A^+=A-C^*C$ (resp. $A^-=-A-C^*C$) is the generator of an exponentially stable $C_0$-semigroup $\T^+$ (resp. $\T^-$). We set as usually 
$$
\Lb_\tau = \T_\tau^-\T_\tau^+.
$$
Throughout this section we always assume that $(w_0,w_1) \in \Ds = \Dztdemi \times \Dz$. Thus by applying Theorem 4.1.6 of Tucsnak and Weiss \cite{TucWei09}, we have $$w \in C\left([0,\tau],\Dztdemi\right) \cap C^1\left([0,\tau],\Dz\right) \cap C^2\left([0,\tau],\Dzdemi\right).$$
The forward and backward observers \eqref{eq-sys-obs-forward} and \eqref{eq-sys-obs-backward} read then as follows (as second-order systems) 
\begin{equation}\label{eq-sys-obs-forward-wave}
\left\lbrace \begin{array}{ll}
\ddot{w}^+(t) + A_0w^+(t) + C_0^*C_0\dot{w}^+(t) = C_0^*y(t), \quad \forall t \in [0,\tau], \\
w^+(0) = 0, \quad \dot{w}^+(0) = 0,
\end{array} \right.
\end{equation}
\begin{equation}\label{eq-sys-obs-backward-wave}
\left\lbrace \begin{array}{ll}
\ddot{w}^-(t) + A_0w^-(t) - C_0^*C_0\dot{w}^-(t) = -C_0^*y(t), \quad \forall t \in [0,\tau], \\
w^-(\tau) = w^+(\tau), \quad \dot{w}^-(\tau) = \dot{w}^+(\tau).
\end{array} \right.
\end{equation}
Clearly, the above two systems can be written as a general initial value Cauchy problem of the same form (simply by using a time reversal for the second system)
\begin{equation}\label{eq-general-q-wave}
\left\{\begin{array}{ll}
\ddot{p}(t) + A_0p(t) + C_0^*C_0 \dot p(t) = f(t), \qquad \forall t \in [0,\tau],\\
p(0) = p_0, \quad \dot p(0) = p_1
\end{array}\right.
\end{equation}
where we have set 
\begin{itemize}
\item for the forward observer \eqref{eq-sys-obs-forward-wave} : $f(t) = C_0^*y (t)= C_0^*C_0 \dot w(t)$ and $(p_0,p_1)=(0,0)$,
\item for the backward observer \eqref{eq-sys-obs-backward-wave} : $f(t) = -C_0^*y(\tau-t) = -C_0^*C_0 \dot w(\tau-t)$ and $(p_0,p_1)=(w^+(\tau),-\dot w^+(\tau)) \in \Ds = \Dztdemi \times \Dz$.
\end{itemize}
Let us emphasize that with these notation, the semigroups $\T^\pm$ are given by the relations
\begin{equation}\label{eq-Tpm-wave}
\T^+_t \left[\begin{matrix} p_0 \\ p_1 \end{matrix}\right] = \left[\begin{matrix} p(t) \\ \dot p(t) \end{matrix}\right]
\qquad\qquad
\T^-_t \left[\begin{matrix} p_0 \\ p_1 \end{matrix}\right] = \left[\begin{matrix} p(\tau-t) \\ -\dot p(\tau-t) \end{matrix}\right] 
\end{equation}
where $p$ solves \eqref{eq-general-q-wave} with $f=0$.

In the next two subsections, we propose a convergence analysis of semi-discretized and fully discretized approximation schemes for the forward and backward observers \eqref{eq-sys-obs-forward-wave} and \eqref{eq-sys-obs-backward-wave}. Our proof is based on the convergence analysis of the semi and fully discretizations of \eqref{eq-general-q-wave}.  As far as we know, the existing literature on the convergence analysis of full discretizations of wave-type systems concern only the particular cases of conservative systems (i.e. without damping), see e.g. Raviart and Thomas \cite[p.\, 197]{RavTho98} or Dautray and Lions \cite[p.\, 921]{DauLio85} and systems with constant damping coefficients Geveci and Kok \cite{GevKok85}. For a recent review of numerical approximation issues related to  the control and the observation of waves, we refer the reader to the review paper of Zuazua \cite{Zua05}.

\subsection{Space Semi-Discretization}
\label{Subsct_Wave1}
\subsubsection{Statement of the main result}
We use a Galerkin method to approximate system \eqref{eq-general-q-wave}. More precisely, consider a family $(H_h)_{h>0}$ of finite-dimensional subspaces of $\Dzdemi$ endowed with the norm in $H$. We denote $\pi_h$ the orthogonal projection from $\Dzdemi$ onto $H_h$. We assume that there exist $M>0$, $\theta>0$ and $h^*>0$ such that we have for all $h\in(0,h^*)$ 
\begin{equation}\label{eq-est-H}
\left\| \pi_h\varphi-\varphi \right\| \leq Mh^\theta \left\| \varphi \right\|_\frac{1}{2}, \qquad \forall \varphi \in \Dzdemi.
\end{equation}
Given $(p_0,p_1) \in \Ds$, the variational formulation of \eqref{eq-general-q-wave} reads for all $t\in[0,\tau]$ and all $\varphi \in \Dzdemi$ as follows
\begin{equation}\label{eq-general-q-FV-wave}
\left\{\begin{array}{ll}
\left\langle\ddot{p}(t),\varphi \right\rangle + \left\langle p(t),\varphi \right\rangle_\frac{1}{2} + \left\langle C_0^* C_0 \dot p(t), \varphi \right\rangle = \left\langle f(t),\varphi \right\rangle, \qquad \forall t \in [0,\tau],\\
p(0) = p_0, \quad \dot p(0) = p_1.
\end{array}\right.
\end{equation}
Suppose that $(p_{0,h},p_{1,h}) \in H_h \times H_h$ and $f_h$ are given approximations of $(p_0,p_1)$ and $f$ respectively in the spaces $X$ and $L^1\left([0,\tau],H\right)$. We define $p_h(t)$ as the solution of the variational problem 
\begin{equation}\label{eq-general-qh-FV-wave}
\left\{\begin{array}{ll}
\left\langle\ddot{p}_h(t),\varphi_h \right\rangle + \left\langle p_h(t),\varphi_h \right\rangle_\frac{1}{2} + \left\langle C_0^* C_0 \dot p_h(t), \varphi_h \right\rangle = \left\langle f_h(t),\varphi_h \right\rangle, \qquad \forall t \in [0,\tau],\\
p_h(0) = p_{0,h}, \quad \dot p_h(0) = p_{1,h}.
\end{array}\right.
\end{equation}
for all $t \in [0,\tau]$ and all $\varphi_h \in H_h$.

The above approximation procedure leads in particular to the definition of the semi-discretized versions $\T^\pm_h$ of the semigroups $\T^\pm$ that we will use. Indeed, we simply set 
\begin{equation}\label{eq-Tpm-h-wave}
\T^+_{h,t} \left[\begin{matrix} p_0 \\ p_1 \end{matrix}\right] = \left[\begin{matrix} p_h(t) \\ \dot p_h(t) \end{matrix}\right]
\qquad\qquad
\T^-_{h,t} \left[\begin{matrix} p_0 \\ p_1 \end{matrix}\right] = \left[\begin{matrix} p_h(\tau-t) \\ -\dot p_h(\tau-t) \end{matrix}\right] 
\end{equation}
where $p_h$ solves \eqref{eq-general-qh-FV-wave} for $f_h=0$ and $(p_{0,h},p_{1,h})=(\pi_hp_0, \pi_hp_1)$.
The semi-discretized counterpart of $\Lb_\tau = \T_\tau^-\T_\tau^+$ is then given by 
$$
\Lb_{h,\tau} = \T^-_{h,\tau}\T^+_{h,\tau}.
$$ 

Assume that $y_h$ is an approximation of the output $y$ in $L^1([0,\tau],Y)$ and let $w_h^+$ and $w_h^-$ denote the Galerkin approximations of the solutions of systems \eqref{eq-sys-obs-forward-wave} and \eqref{eq-sys-obs-backward-wave}, satisfying for all $t \in [0,\tau]$ and all $\varphi_h \in H_h$
\begin{equation}\label{eq-sys-obs-forward-h-wave}
\left\lbrace \begin{array}{ll}
\left\langle\ddot{w}_h^+(t),\varphi_h\right\rangle + \left\langle w_h^+(t),\varphi_h\right\rangle_\frac{1}{2} + \left\langle C_0^*C_0\dot{w}_h^+(t),\varphi_h\right\rangle = \left\langle C_0^*y_h(t),\varphi_h\right\rangle, \\
w_h^+(0) = 0, \quad \dot{w}_h^+(0) = 0,
\end{array} \right.
\end{equation}
\begin{equation}\label{eq-sys-obs-backward-h-wave}
\left\lbrace \begin{array}{ll}
\left\langle\ddot{w}_h^-(t),\varphi_h\right\rangle + \left\langle w_h^-(t),\varphi_h\right\rangle_\frac{1}{2} - \left\langle C_0^*C_0\dot{w}_h^-(t),\varphi_h\right\rangle = -\left\langle C_0^*y_h(t),\varphi_h\right\rangle, \\
w_h^-(\tau) = w_h^+(\tau), \quad \dot{w}_h^-(\tau) = \dot{w}_h^+(\tau).
\end{array} \right.
\end{equation}

With the above notation, the main result of this section reads as follows.
\begin{theorem}\label{th-main-wave}
Let $A_0 : \Dz \rightarrow H$ be a strictly positive self-adjoint operator and $C_0 \in \mathcal{L}(H,Y)$ such that $C_0^*C_0 \in \mathcal{L}\left(\Dztdemi\right)\cap\mathcal{L}\left(\Dzdemi\right)$. Define $(A,C)$ by \eqref{eq-def-A} and \eqref{eq-def-C}. Assume that the pair $(A,C)$ is exactly observable in time $\tau>0$ and set $\eta:=\|\Lb_\tau\|_{\mathcal{L}(X)} < 1$. Let $(w_0,w_1) \in \Dztdemi \times \Dz$ be the initial value of \eqref{eq-sys-initial-wave} and let $(w_{0,h},w_{1,h})$ be defined by 
\begin{equation}\label{eq-w0hw1h-neumann}
\left[\begin{matrix} w_{0,h} \\ w_{1,h} \end{matrix}\right] = \sum_{n=0}^{N_h} \Lb_{h,\tau}^{n} \left[\begin{matrix} w_h^-(0) \\ \dot w_h^-(0) \end{matrix}\right].
\end{equation}
Then there exist $M > 0$ and $h^* > 0$ such that for all $h \in (0, h^*)$ 
\begin{multline*}
\| w_0 - w_{0,h} \|_\frac{1}{2} + \| w_1 - w_{1,h} \| \leq M \Bigg[ \left(\frac{\eta^{N_h+1}}{1-\eta} + h^\theta \tau N_h^2 \right)\left(\| w_0 \|_\frac{3}{2} + \| w_1 \|_1\right) \\
 + N_h \int_0^\tau \| C_0^*\left(y(s)-y_h(s)\right) \|ds \Bigg].
\end{multline*}
\end{theorem}
\begin{corollary}
Under the assumptions of Theorem \ref{th-main-wave}, we set 
$$
N_h = \theta \frac{\ln h}{\ln \eta}.
$$
Then, there exist $ M_\tau > 0 $ and $ h^* > 0 $ such that for all $h \in (0, h^*)$
\begin{multline}\label{eq-w0w0h-w1w1h}
\| w_0 - w_{0,h} \|_\frac{1}{2} + \| w_1 - w_{1,h} \| \leq M_\tau \Bigg[h^\theta \ln^2 h\, \left(\| w_0 \|_\frac{3}{2} + \| w_1 \|_1\right) \\
+ |\ln h| \int_0^\tau \| C_0^*\left(y(s)-y_h(s)\right) \|ds \Bigg].
\end{multline}
\end{corollary}

\subsubsection{Proof of Theorem \ref{th-main-wave}}
The next Proposition provides the error estimate for the approximation of \eqref{eq-general-q-FV-wave} by using the Galerkin scheme \eqref{eq-general-qh-FV-wave}.
\begin{proposition}\label{prop-pihq-qh-wave}
Given $(p_0,p_1) \in \Dztdemi \times \Dz$ and $(p_{0,h},p_{1,h}) \in H_h \times H_h$, let $p$ and $p_h$ be the solutions of \eqref{eq-general-q-FV-wave} and \eqref{eq-general-qh-FV-wave} respectively. Assume that $C_0^*C_0 \in \mathcal{L}\left(\Dzdemi\right)$. Then, there exist $M>0$ and $h^*>0$ such that for all $t\in[0,\tau]$ and all $h\in(0,h^*)$ 
\begin{multline*}
\| \pi_hp(t)-p_h(t) \|_\frac{1}{2} + \| \pi_h\dot p(t)-\dot p_h(t) \| \leq M\bigg\{\| \pi_hp_0-p_{0,h} \|_\frac{1}{2} + \| \pi_hp_1-p_{1,h} \|\\
 + h^\theta \left[ t\left( \| p_0 \|_\frac{3}{2} + \| p_1 \|_1 + \| f \|_{\frac{1}{2},\infty} \right) + t^2 \| f \|_{1,\infty} \right]\bigg\} + \int_0^t \| f(s)-f_h(s) \| ds.
\end{multline*}
\end{proposition}
\begin{proof}
First, we substract \eqref{eq-general-qh-FV-wave} from \eqref{eq-general-q-FV-wave} to obtain (we omit the time dependence for the sake of clarity) for all $\varphi_h\in H_h$ 
$$
\left\langle \ddot p - \ddot{p}_h,\varphi_h \right\rangle + \left\langle p - p_h,\varphi_h \right\rangle_\frac{1}{2} + \left\langle C_0^* C_0\left( \dot p - \dot p_h \right), \varphi_h \right\rangle = \left\langle f-f_h,\varphi_h \right\rangle.
$$
Noting that $\left\langle \pi_hp-p,\varphi_h \right\rangle_\frac{1}{2} = 0$ for all $\varphi_h\in H_h$ and that $\pi_h\ddot{p}$ makes sense by the regularity of $p$ (this is a direct consequence of relation \eqref{eq-reg-q} from Lemma \ref{lem-annexe2} used with $q=\left[\begin{matrix}p\\\dot p\end{matrix}\right]$
), we obtain from the above equality that for all $\varphi_h\in H_h$ 
\begin{multline}\label{eq-tech-pihq-qh-wave}
\left\langle \pi_h\ddot{p}-\ddot{p}_h,\varphi_h \right\rangle + \left\langle \pi_hp-p_h,\varphi_h \right\rangle_\frac{1}{2} = \left\langle \pi_h\ddot{p}-\ddot{p},\varphi_h \right\rangle + \left\langle C_0^*C_0\left(\dot p_h - \dot p\right),\varphi_h \right\rangle + \left\langle f-f_h,\varphi_h \right\rangle.
\end{multline}
On the other hand, setting 
$$
\mathcal{E}_h = \frac{1}{2} \|\pi_h \dot p- \dot p_h\|^2 + \frac{1}{2} \|\pi_hp-p_h\|_\frac{1}{2}^2,
$$
we have 
$$
\dot{\mathcal{E}}_h = \left\langle \pi_h\ddot{p}-\ddot{p}_h,\pi_h\dot{p}-\dot{p}_h \right\rangle + \left\langle \pi_hp-p_h,\pi_h\dot{p}-\dot{p}_h \right\rangle_\frac{1}{2}.
$$
Applying \eqref{eq-tech-pihq-qh-wave} with $\varphi_h=\pi_h \dot p- \dot p_h$ and substituting the result in the above relation, we obtain by using Cauchy-Schwarz inequality and the boundedness of $C_0$ that there exists $M>0$ such that 
$$
\dot{\mathcal{E}}_h \leq \Big( \|\pi_h\ddot{p}-\ddot{p}\| + M \|\pi_h \dot p - \dot p\| + \|f-f_h\| \Big) \underbrace{\|\pi_h \dot p- \dot p_h\|}_{\leq \sqrt{2\mathcal{E}_h}}.
$$
Since $ \dfrac{\dot{\mathcal{E}}_h}{\sqrt{2\mathcal{E}_h}}=\dfrac{d}{dt}\sqrt{2\mathcal{E}_h}$, the integration of the above inequality from $0$ to $t$ yields 
\begin{multline}\label{eq-intermediate-wave}
\|\pi_h p(t)-p_h(t)\|_\frac{1}{2} + \|\pi_h \dot p(t)-\dot p_h(t)\| \leq M\Bigg\{\|\pi_hp_0-p_{0,h}\|_\frac{1}{2} + \|\pi_hp_1-p_{1,h}\| \\
+ \int_0^t \left(\|\pi_h\ddot{p}(s)-\ddot{p}(s)\| + \|\pi_h\dot p(s)- \dot p(s)\|\right)ds +\int_0^t\|f(s)-f_h(s)\|ds\Bigg\}.
\end{multline}
Thus, it remains to bound $\|\pi_h\ddot{p}(t)-\ddot{p}(t)\|$ and $\|\pi_h \dot p(t)- \dot p(t)\|$ for all $t\in [0,\tau]$. Using \eqref{eq-est-H} and the classical continuous embedding from ${\cal D}(A^\alpha)$ to ${\cal D}(A^\beta)$ for $\alpha>\beta$, we get that 
$$
\left\{
\begin{array}{l}
\|\pi_h\ddot{p}(t)-\ddot{p}(t)\| \leq M h^\theta \|\ddot{p}(t)\|_\frac{1}{2},\\
\|\pi_h \dot p(t)- \dot p(t)\| \leq M h^\theta \|\dot p(t)\|_\frac{1}{2} \leq M h^\theta \| \dot p(t)\|_1,
\end{array}
\right.
\qquad \forall t \in [0,\tau], \; h \in (0,h^*).
$$
Using relations \eqref{eq-dotq-norm-1} proved in Lemma \ref{lem-annexe2} of the Appendix for the first order unknown $q=\left[\begin{matrix}p \\ \dot p\end{matrix}\right]$ and the right-hand side $F=\left[\begin{matrix}0 \\ f\end{matrix}\right]$, we get for all $t \in [0,\tau]$ and all $h \in (0,h^*)$ 
$$
\|\pi_h\ddot{p}(t)-\ddot{p}(t)\| + \|\pi_h \dot p(t)- \dot p(t)\| \leq M h^\theta \left( \|p_0\|_\frac{3}{2} + \|p_1\|_1 + t \|f\|_{1,\infty} + \|f\|_{\frac{1}{2},\infty} \right).
$$
Substituting the above inequality in \eqref{eq-intermediate-wave}, we get the result.
\end{proof}
Thanks to the last result, we are now in position to derive an error approximation for the semigroups $\T^\pm$ and for the operator $\Lb_t=\T_t^-\T_t^+$. This result has been recently proved in the preprint \cite{CinMicTuc10} but we prefer to include the proof for the sake of completeness and clarity.
\begin{proposition}\label{prop-L-tau-h-wave}
Let $\Pi_h = \left[\begin{matrix}\pi_h & 0\\0 & \pi_h\end{matrix}\right]$. Under the assumptions of Proposition \ref{prop-pihq-qh-wave}, the following assertions hold true
\begin{enumerate}
\item There exist $M>0$ and $h^*>0$ such that for all $t\in(0,\tau)$ and all $h\in(0,h^*)$ 
\begin{equation}\label{eq-cor-pihq-qh-forward-wave}
\left\|(\Pi_h\T^+_t - \T^+_{h,t}) \Pz  \right\| \leq M t h^\theta \left(\|p_0\|_\frac{3}{2} + \|p_1\|_1\right),
\end{equation}
\begin{equation}\label{eq-cor-pihq-qh-backward-wave}
\left\|(\Pi_h\T^-_t - \T^-_{h,t} ) \Pz \right\| \leq M (\tau-t) h^\theta \left(\|p_0\|_\frac{3}{2} + \|p_1\|_1\right).
\end{equation}

\item There exist $M>0$ and $h^*>0$ such that for all $n\in\mathbb{N}$, all $t\in[0,\tau]$ and all $h\in(0,h^*)$, we have 
\begin{equation}\label{eq-cor-pihq-qh-wave2}
\left\| (\Lb_t^n - \Lb_{h,t}^n) \Pz \right\| \leq M(1 + n \tau ) h^\theta \left(\|p_0\|_\frac{3}{2} + \|p_1\|_1\right).
\end{equation}
\end{enumerate}
\end{proposition}
\begin{proof}\ 

1. From relations \eqref{eq-Tpm-wave} and \eqref{eq-Tpm-h-wave} defining the continuous and the semi-discretized semigroups $\T^\pm$ and $\T^\pm_h$, \eqref{eq-cor-pihq-qh-forward-wave} and \eqref{eq-cor-pihq-qh-backward-wave}  follow immediately from Proposition \ref{prop-pihq-qh-wave}. 

2. Let $t\in [0,\tau]$ and $(p_0,p_1) \in \Dztdemi \times \Dz$. We have
\begin{equation}\label{eq-Lt-ineq-tri-wave}
\left\| (\Lb^n_t - \Lb^n_{h,t})\left[\begin{matrix}p_0 \\ p_1\end{matrix}\right] \right\| \leq \left\| (\Lb^n_t  - \Pi_h\Lb^n_t) \left[\begin{matrix}p_0 \\ p_1\end{matrix}\right] \right\| + \left\| (\Pi_h\Lb^n_t - \Lb^n_{h,t}) \left[\begin{matrix}p_0 \\ p_1\end{matrix}\right] \right\|.
\end{equation}
Using \eqref{eq-est-H} and the fact that $\|\Lb_t\|_{\mathcal{L}(\D)} \leq 1$ proved in Lemma \ref{lem-contraction} of the Appendix, the first term in the above relation can be estimated as follows 
\begin{equation}\label{eq-Lt-first-term-wave}
\left\| (\Lb^n_t  - \Pi_h\Lb^n_t) \left[\begin{matrix}p_0 \\ p_1\end{matrix}\right] \right\| \leq Mh^\theta \left(\|p_0\|_\frac{3}{2} + \|p_1\|_1\right), \quad \forall h \in (0,h^*).
\end{equation}
For the second term in \eqref{eq-Lt-ineq-tri-wave}, let us prove by induction that for all $n\in\mathbb{N}$
\begin{equation}\label{eq-Lt-second-term-wave}
\left\| (\Pi_h\Lb^n_t - \Lb^n_{h,t}) \left[\begin{matrix}p_0 \\ p_1\end{matrix}\right] \right\| \leq M n\tau h^\theta \left(\|p_0\|_\frac{3}{2} + \|p_1\|_1\right), \quad \forall h \in (0,h^*), \forall t\in[0,\tau].
\end{equation}
We have
$$
\begin{array}{rcl}
\dsp \left\| (\Pi_h\Lb_t  - \Lb_{h,t}) \Pz \right\| &= &\dsp \left\|( \Pi_h\T^-_t\T^+_t - \T_{h,t}^-\T_{h,t}^+ )\Pz \right\| \\
 &\leq & \dsp \ \left\| (\Pi_h\T^-_t  - \T_{h,t}^-)\T_t^+ \Pz \right\| + \left\| \T_{h,t}^-(\T_t^+ - \T_{h,t}^+)\Pz \right\|.
\end{array}
$$
By Lemma \ref{lem-contraction} of the Appendix and equation \eqref{eq-cor-pihq-qh-backward-wave}, we get 
$$
\left\| (\Pi_h\T^-_t - \T_{h,t}^-)\T_t^+  \Pz \right\| \leq M (\tau-t) h^\theta \left(\|p_0\|_\frac{3}{2} + \|p_1\|_1\right), \quad \forall h \in (0,h^*).
$$
Obviously $\|\T^-_h\|_{\mathcal{L}(X)}$ is uniformly bounded (with respect to $h$), and thus by \eqref{eq-est-H} and equations \eqref{eq-cor-pihq-qh-forward-wave}-\eqref{eq-cor-pihq-qh-backward-wave}, we have (by using the triangle inequality)
$$
\left\| \T_{h,t}^-(\T_t^+  - \T_{h,t}^+ ) \Pz \right\| \leq M t h^\theta \left(\|p_0\|_\frac{3}{2} + \|p_1\|_1\right), \quad \forall h \in (0,h^*).
$$
Consequently
\begin{equation}\label{eq-pih-Lt-wave}
\left\| (\Pi_h\Lb_t - \Lb_{h,t}) \Pz \right\| \leq M \tau h^\theta \left(\|p_0\|_\frac{3}{2} + \|p_1\|_1\right), \quad \forall h \in (0,h^*), \forall t\in[0,\tau]
\end{equation}
which shows that \eqref{eq-Lt-second-term-wave} holds for $n=1$. Suppose now that for a given $n\geq 2$, there holds 
\begin{equation}\label{eq-pih-Lt-rec-wave}
\left\| (\Pi_h\Lb^{n-1}_t  - \Lb^{n-1}_{h,t}) \Pz \right\| \leq M (n-1) \tau h^\theta \left(\|p_0\|_\frac{3}{2} + \|p_1\|_1\right).
\end{equation}
From
$$
\left\| (\Pi_h\Lb^n_t - \Lb^n_{h,t}) \Pz \right\| \leq \left\| (\Pi_h \Lb_t \Lb^{n-1}_t - \Lb_{h,t} \Lb^{n-1}_t) \Pz \right\| + \left\| \Lb_{h,t} (\Lb^{n-1}_t - \Lb^{n-1}_{h,t}) \Pz \right\|.
$$
we obtain by using \eqref{eq-pih-Lt-wave} and \eqref{eq-pih-Lt-rec-wave} that (thanks to Lemma \ref{lem-contraction} and the uniform boundedness of $\|\Lb_{h,t}\|_{\mathcal{L}(X)}$ with respect to $h$)
$$
\left\| (\Pi_h\Lb^n_t - \Lb^n_{h,t}) \Pz \right\| \leq M n\tau  h^\theta \left(\|p_0\|_\frac{3}{2} + \|p_1\|_1\right),
$$
which is exactly \eqref{eq-Lt-second-term-wave}. Substituting \eqref{eq-Lt-first-term-wave} and \eqref{eq-Lt-second-term-wave} in \eqref{eq-Lt-ineq-tri-wave}, we obtain the result.
\end{proof}

Now, we can turn to the proof of Theorem \ref{th-main-wave} 
\begin{proof}[of Theorem \ref{th-main-wave}]
Introducing the term $\dsp \sum_{n=0}^{N_h} \Lb_{h,\tau}^n \left[\begin{matrix}w^-(0) \\ \dot w^-(0)\end{matrix}\right] $, we first rewrite the error term $\dsp \left[\begin{matrix}w_0 \\ w_1\end{matrix}\right] - \left[\begin{matrix}w_{0,h} \\ w_{1,h}\end{matrix}\right]=\sum_{n=0}^\infty \Lb_\tau^n \left[\begin{matrix}w^-(0) \\ \dot w^-(0)\end{matrix}\right] -\sum_{n=0}^{N_h} \Lb_{h,\tau}^n \left[\begin{matrix}w_h^-(0) \\ \dot w_h^-(0)\end{matrix}\right]$ in the following form
\begin{multline*}
\left[\begin{matrix}w_0 \\ w_1\end{matrix}\right] - \left[\begin{matrix}w_{0,h} \\ w_{1,h}\end{matrix}\right] 
=
\sum_{n>N_h} \Lb_\tau^n \left[\begin{matrix}w^-(0) \\ \dot w^-(0)\end{matrix}\right] + \sum_{n=0}^{N_h}\left(\Lb_\tau^n - \Lb_{h,\tau}^n\right) \left[\begin{matrix}w^-(0) \\ \dot w^-(0)\end{matrix}\right]\\
\hspace{8cm}+\sum_{n=0}^{N_h} \Lb_{h,\tau}^n \left[\begin{matrix}w^-(0)-w_h^-(0) \\ \dot w^-(0) - \dot w_h^-(0)\end{matrix}\right].\hspace{5.5cm}
\end{multline*}
Therefore, we have 
\begin{equation}\label{eq-diff-z0-z0h-wave}
\left\|\left[\begin{matrix}w_0 \\ w_1\end{matrix}\right] - \left[\begin{matrix}w_{0,h} \\ w_{1,h}\end{matrix}\right] \right\|
\leq S_1+S_2+S_3,
\end{equation}
where we have set 
$$
\left\{
\begin{array}{lll}
\dsp S_1= \sum_{n>N_h} \left\| \Lb_\tau^n \left[\begin{matrix}w^-(0) \\ \dot w^-(0)\end{matrix}\right] \right\|,\\
\dsp S_2=\sum_{n=0}^{N_h} \left\| \left(\Lb_\tau^n - \Lb_{h,\tau}^n\right) \left[\begin{matrix}w^-(0) \\ \dot w^-(0)\end{matrix}\right] \right\|,\\
\dsp S_3=\left(\sum_{n=0}^{N_h} \left\| \Lb_{h,\tau}^n \right\|_{\mathcal{L}(X)}\right) \left\|  \left[\begin{matrix}w^-(0) \\ \dot w^-(0)\end{matrix}\right] \right\|.
\end{array} 
\right.
$$
Note that the term $S_1$ is the truncation error of the tail of the infinite sum \eqref{eq-z0-neumann}, the term $S_2$ represents the cumulated error due to the approximation of the semigroups $\T^\pm$ while the term $S_3$ comes from the approximation of the first iterate $\left[\begin{matrix}w^-(0) \\ \dot w^-(0)\end{matrix}\right]$ of the reconstruction algorithm. 

Since $\eta=\|\Lb_\tau\|_{\mathcal{L}(X)} < 1$, the first term can be estimated very easily using relation \eqref{eq-z0-Ltau}:
\begin{equation}\label{eq-first-term-th-wave}
S_1\leq M\, \frac{\eta^{N_h+1}}{1-\eta} \left( \| w_0 \|_\frac{3}{2} + \| w_1 \|_1 \right).
\end{equation}
The term $S_2$ can be estimated using the estimate \eqref{eq-cor-pihq-qh-wave2} from Proposition \ref{prop-L-tau-h-wave} 
$$
S_2 \leq M \left(\sum_{n=0}^{N_h} (1+n \tau)\right) h^\theta \left( \|w^-(0)\|_\frac{3}{2} + \|\dot w^-(0)\|_1 \right), \quad \forall h \in (0,h^*).
$$
Therefore, using \eqref{eq-z0-Ltau} and the fact that $\|\Lb_\tau\|_{\Ds}<1$ (see Lemma \ref{lem-contraction}) in the above relation, we finally get that 
\begin{equation}\label{eq-second-term-th-wave}
S_2 \leq M\Big[1+(1+\tau)N_h+N_h^2\tau \Big] h^\theta \left( \| w_0 \|_\frac{3}{2} + \| w_1 \|_1 \right), \quad \forall h \in (0,h^*).
\end{equation}
Finally, let us estimate the term $S_3$. As $\|\Lb_{h,\tau}\|_{\mathcal{L}(X)}$ is uniformly bounded by 1 with respect to $h$, we have 
\begin{equation}\label{eq-S3-num1-wave}
\begin{array}{rcl}
S_3 &\leq & \dsp M N_h \left(\| w^-(0) - w_h^-(0) \|_\frac{1}{2} + \| \dot w^-(0) - \dot w_h^-(0) \| \right)\\
 & \leq &MN_h \dsp\Big(\left\| w^-(0) - \pi_h w^-(0)\right\|_\frac{1}{2} + \left\|\pi_h w^-(0)- w_h^-(0)\right\|_\frac{1}{2} \\
& & + \left\| \dot w^-(0) - \pi_h \dot w^-(0)\right\| + \left\|\pi_h \dot w^-(0)- \dot w_h^-(0)\right\| \Big).
\end{array}
\end{equation}
By using \eqref{eq-est-H} and \eqref{eq-z0-Ltau}, we immediately obtain that
\begin{equation}\label{eq-S3-num2-wave}
\| w^-(0) - \pi_h w^-(0) \|_\frac{1}{2} + \| \dot w^-(0) - \pi_h \dot w^-(0) \| \leq Mh^\theta \left( \| w_0 \|_\frac{3}{2} + \| w_1 \|_1 \right).
\end{equation}
To estimate $\| \pi_h w^-(0)- w_h^-(0) \|_\frac{1}{2} + \| \pi_h \dot w^-(0) - \dot w_h^-(0) \|$, we apply twice Proposition \ref{prop-pihq-qh-wave} first for the time reversed backward observer $w^-(\tau-\cdot)$ and then for the forward observer $w^+$ (the time reversal is introduced just because Proposition \ref{prop-pihq-qh-wave} can only be applied to initial value Cauchy problems). After straightforward calculation we obtain that for all $h\in(0,h^*)$ 
\begin{multline}\label{eq-intermediate-main-1-wave}
\| \pi_h w^-(0) - w_h^-(0) \|_\frac{1}{2} + \| \pi_h \dot w^-(0) - \dot w_h^-(0) \| \\
 \leq M h^\theta\Big[ \tau(\|w^+(\tau)\|_\frac{3}{2} + \| \dot w^+(\tau) \|_1 + \|C_0^*y\|_{\frac{1}{2}, \infty}) + \tau^2\|C_0^*y\|_{1,\infty} \Big]\\
\hspace{6cm}+ \int_0^\tau\|C_0^*\left(y(s)-y_h(s)\right)\|ds.
\end{multline}
Applying \eqref{eq-q-norm-1-2} of Lemma \ref{lem-annexe2} of the Appendix with zero initial data, we obtain that 
$$
\|w^+(\tau)\|_\frac{3}{2} + \|\dot w^+(\tau)\|_1 \leq \tau \|C_0^*y\|_{1,\infty}.
$$
Therefore \eqref{eq-intermediate-main-1-wave} also reads
\begin{multline*}
\| \pi_h w^-(0) - w_h^-(0) \|_\frac{1}{2} + \| \pi_h \dot w^-(0) - \dot w_h^-(0) \| \leq M h^\theta(\tau + \tau^2)\|C_0^*y\|_{1,\infty} \\
+ \int_0^\tau\|C_0^*\left(y(s)-y_h(s)\right)\|ds.\hspace{0.5cm}
\end{multline*}
As $C_0^*C_0 \in \mathcal{L}\left(\Dztdemi\right) \cap \mathcal{L}\left(\Dzdemi\right)$ and $\|w\|_{\frac{3}{2},\infty} + \|\dot w\|_{1,\infty} = \|w_0\|_\frac{3}{2} + \|w_1\|_1$ (since $A$ is skew-adjoint), the last relation becomes 
\begin{multline*}
\| \pi_h w^-(0) - w_h^-(0) \|_\frac{1}{2} + \| \pi_h \dot w^-(0) - \dot w_h^-(0) \| \leq M h^\theta(\tau + \tau^2)\left(\|w_0\|_\frac{3}{2} + \|w_1\|_1\right) \\
+ \int_0^\tau\|C_0^*\left(y(s)-y_h(s)\right)\|ds.
\end{multline*}
Using the above relation and \eqref{eq-S3-num2-wave} in \eqref{eq-S3-num1-wave}, we get
\begin{equation}\label{eq-third-term-th-wave}
S_3 \leq M N_h \left(h^\theta (1+\tau + \tau^2)\left(\|w_0\|_\frac{3}{2} + \|w_1\|_1\right) + \int_0^\tau\|C_0^*\left(y(s)-y_h(s)\right)\|ds \right).
\end{equation}
Substituting \eqref{eq-first-term-th-wave}, \eqref{eq-second-term-th-wave} and \eqref{eq-third-term-th-wave} in \eqref{eq-diff-z0-z0h-wave}, we get for all $h\in(0,h^*)$ 
\begin{multline*}
\| w_0 - w_{0,h} \|_\frac{1}{2} + \| w_1 - w_{1,h} \| \leq \\
M \Bigg[ \bigg(\frac{\eta^{N_h+1}}{1-\eta} +h^\theta \left[1+(1+\tau+\tau^2)N_h+\tau N_h^2 \right] \bigg)\left(\|w_0\|_\frac{3}{2} + \|w_1\|_1\right) \\
+ N_h \int_0^\tau \| C_0^*\left(y(s)-y_h(s)\right) \|ds \Bigg],
\end{multline*}
which leads to the result (with possibly reducing the value of $h^*$).
\end{proof}
\subsection{Full Discretization}
\label{Subsct_Wave2}
\subsubsection{Statement of the main result}
In order to approximate \eqref{eq-general-q-FV-wave} in space and time, we use a finite difference scheme in time combined with the previous Galerkin approximation in space. We discretize the time interval $[0,\tau]$ using a time step $\Delta t>0$. We obtain a discretization $t_k = k\Delta t$, where $0\leq k \leq K$ and where we assumed, without loss of generality, that $\tau =K\Delta t$. Given a function of time $f$ of class ${\cal C}^2$, we approximate its first and second derivative at time $t_k$ by 
$$
f'(t_k) \simeq D_tf(t_k) := \frac{f(t_k)-f(t_{k-1})}{\Delta t}.
$$
$$
f''(t_k) \simeq D_{tt}f(t_k) := \frac{f(t_k)-2f(t_{k-1})+f(t_{k-2})}{\Delta t^2}.
$$
We suppose that $(p_{0,h,\Delta t},p_{1,h,\Delta t})\in H_h \times H_h$ and $f^k_{h}$, for $0\le k \le K$, are given approximations of $(p_0,p_1)$ and $f(t_k)$ in the space $X$ and $H$ respectively. We define the approximate solution $(p_{h}^k)_{0\leq k \leq K}$ of \eqref{eq-general-q-FV-wave} as the solution of the following problem: $p_{h}^k\in H_h$ such that for all $\varphi_h \in H_h$
\begin{equation}\label{eq-general-qhk-FVDF-wave}
\left\{\begin{array}{ll}
\left\langle D_{tt}p_h^k,\varphi_h\right\rangle + \left\langle p_h^k,\varphi_h \right\rangle_\frac{1}{2} + \left\langle C_0^*C_0D_tp_h^k,\varphi_h \right\rangle = \left\langle f_h^k,\varphi_h \right\rangle,\qquad 2\le k \le K\\
p_h^0=p_{0,h,\Delta t}, \quad p_h^1 = p_h^0 + \Delta t \, p_{1,h,\Delta t}.
\end{array}\right.
\end{equation}
Note that the above procedure leads to a natural approximation $\T^\pm_{h,\Delta t,k}$ of the continuous operators $\T^\pm_{t_k}$ by setting
\begin{equation}\label{eq-Tpm-h-Deltat-wave}
\left\{
\begin{array}{l}
\T^+_{t_k} \Pz \simeq \T^+_{h,\Delta t,k} \Pz := \left[ \begin{matrix} p^k_h \\ D_tp^k_h \end{matrix} \right]\\
\\
\T^-_{t_k} \Pz \simeq \T^-_{h,\Delta t,k} \Pz := \left[ \begin{matrix} p^{K-k}_h \\ -D_tp^{K-k}_h \end{matrix} \right]
\end{array}
\right.
\end{equation}
where $p^k_h$ solves \eqref{eq-general-qhk-FVDF-wave} with $f^k_{h}= 0$ for all $0\le k \le K$ and for $(p_{0,h,\Delta t},p_{1,h, \Delta t})=(\pi_h p_0, \pi_h p_1)$. Obviously, this also leads to a fully discretized approximation of the operator  $\Lb_\tau=\T^-_\tau\T^+_\tau$ by setting
$$
\Lb_{h,\Delta t,K} = \T^-_{h,\Delta t,K}\T^+_{h,\Delta t,K}.
$$
Assume that for all $0\le k \le K$, $y_h^k$ is a given approximation of $y(t_k)$ in $Y$ and
 let $\left(w_{h}^+\right)^k$ and $\left(w_{h}^-\right)^k$ be respectively  the approximations of \eqref{eq-sys-obs-forward-wave} and \eqref{eq-sys-obs-backward-wave}  obtained via \eqref{eq-general-qhk-FVDF-wave} as follows:
\begin{itemize}
\item For all $0\le k \le K$, $\left(w_{h}^+\right)^k=p^k_{h}$ where $p^k_{h}$ solves \eqref{eq-general-qhk-FVDF-wave} with  $f^k_{h} = C_0^*y_{h}^k$ and $(p_{0,h,\Delta t},p_{1,h,\Delta t})=(0,0)$,
\item For all $0\le k \le K$, $\left(w_{h}^-\right)^k=p^{K-k}_{h}$ where $p^k_{h}$ solves \eqref{eq-general-qhk-FVDF-wave} with  $f^k_{h} = - C_0^*y_{h}^{K-k}$ and $(p_{0,h,\Delta t},p_{1,h,\Delta t})=((w_{h}^+)^K,-D_t(w_{h}^+)^K)$.
\end{itemize}
Then, our main result (the fully discrete counterpart of Theorem \ref{th-main-wave}) reads as follows 
\begin{theorem}\label{th-main-full-wave}
Let $A_0 : \Dz \rightarrow H$ be a strictly positive self-adjoint operator and $C_0 \in \mathcal{L}(H,Y)$ such that $C_0^*C_0 \in \mathcal{L}\left(\Dztdemi\right)\cap\mathcal{L}\left(\Dzdemi\right)$. Define $(A,C)$ by \eqref{eq-def-A} and \eqref{eq-def-C}. Assume that the pair $(A,C)$ is exactly observable in time $\tau>0$ and set $\eta:=\|\Lb_\tau\|_{\mathcal{L}(X)} < 1$. Let $(w_0,w_1) \in \Dztdemi \times \Dz$ be the initial value of \eqref{eq-sys-initial-wave} and let $(w_{0,h,\Delta t},w_{1,h,\Delta t})$ be defined by 
\begin{equation}\label{eq-w0hDeltatw1hDeltat-neumann}
\left[\begin{matrix} w_{0,h,\Delta t} \\ w_{1,h,\Delta t} \end{matrix}\right] = 
\sum_{n=0}^{N_h} \Lb_{h,\Delta t,K}^{n}  \left[\begin{matrix} (w_h^-)^{0} \\ D_t(w_h^-)^{1}\end{matrix}\right],
\end{equation}
where  $\dsp D_t(w_h^-)^1 = \dfrac{(w_h^-)^{1}-(w_h^-)^{0}}{\Delta t}$.

Then there exist $M > 0$, $h^* > 0$ and $\Delta t^* > 0$ such that for all $h \in (0, h^*)$ and $\Delta t \in (0, \Delta t^*)$
\begin{multline*}
\| w_0 - w_{0,h,\Delta t} \|_\frac{1}{2} + \| w_1 - w_{1,h,\Delta t} \| \\
\leq M \Bigg[ \Bigg(\frac{\eta^{N_{h,\Delta t}+1}}{1-\eta} + \left(h^\theta + \Delta t\right) \left(1+\tau\right) N_{h,\Delta t}^2 \Bigg)\left(\| w_0 \|_\frac{3}{2} + \| w_1 \|_1\right)\\
\hspace{7cm} + N_{h,\Delta t} \Delta t\sum_{\ell=0}^K \left\| C_0^*(y(t_\ell)-y_h^\ell) \right\| \Bigg].
\end{multline*}
\end{theorem}
\begin{corollary}
Under the assumptions of Theorem \ref{th-main-full-wave}, we set 
$$
N_{h,\Delta t} = \frac{\ln (h^\theta + \Delta t)}{\ln \eta}
$$
Then, there exist $ M_\tau > 0 $, $ h^* > 0 $ and $\Delta t^*>0$ such that for all $h \in (0, h^*)$ and $\Delta t\in(0,\Delta t^*)$ 

\begin{multline}\label{eq-w0w0hk-w1w1hk}
\| w_0 - w_{0,h,\Delta t} \|_\frac{1}{2} + \| w_1 - w_{1,h,\Delta t} \| \leq M_\tau \bigg[(h^\theta+\Delta t) \ln^2 (h^\theta+\Delta t) \left(\| w_0 \|_\frac{3}{2} + \| w_1 \|_1\right) \\
+ \left|\ln (h^\theta + \Delta t)\right| \Delta t \sum_{\ell=0}^{K}\left\|C_0^*\left(y(t_\ell)-y_h^\ell\right)\right\| \bigg].
\end{multline}
\end{corollary}
\subsubsection{Proof of Theorem \ref{th-main-full-wave}}
As in the semi-discrete case, the main ingredient for the convergence analysis is the following result (the counterpart of Proposition \ref{prop-pihq-qh-wave}) which gives the error estimate for the full approximation of the general system \eqref{eq-general-q-FV-wave} by \eqref{eq-general-qhk-FVDF-wave}.
\begin{proposition}\label{prop-pihq-qhk-wave}
Given $(p_0,p_1) \in \Dztdemi \times \Dz$ and $(p_{0,h,\Delta t},p_{1,h,\Delta t}) \in H_h \times H_h$, let $p$ and $(p_h^k)_k$ be the solutions of \eqref{eq-general-q-FV-wave} and \eqref{eq-general-qhk-FVDF-wave} respectively. Assume that $C_0^*C_0 \in \mathcal{L}\left(\Dzdemi\right)$. Then, there exist $M>0$, $h^*>0$ and $\Delta t^*>0$ such that for all $1\le k\le K$, all $h\in(0,h^*)$ and all $\Delta t\in(0,\Delta t^*)$
\begin{multline*}
\| \pi_hp(t_k)-p_h^k \|_\frac{1}{2} + \| \pi_h \dot p(t_k)- D_t p_h^k \| \leq M\bigg\{\| \pi_hp_0-p_{0,h,\Delta t} \|_\frac{1}{2} + \| \pi_hp_1-p_{1,h,\Delta t} \|\\
 + \left(h^\theta + \Delta t \right) \left[ t_k\left( \| p_0 \|_\frac{3}{2} + \| p_1 \|_1 + \| f \|_{\frac{1}{2},\infty} + \|\dot{f}\|_\infty \right) + t_k^2 \| f \|_{1,\infty} \right]\\
+ \Delta t \sum_{\ell=1}^{k}\|f(t_\ell)-f_{h}^\ell\|\bigg\}.
\end{multline*}
\end{proposition}
\begin{proof}
Denote by $r_1(t_k)$ the residual term in the first order Taylor expansion of $p$ around $t_{k-1}$. Then
\begin{equation}\label{eq-taylor-wave}
\dot{p}(t_k) = \frac{p(t_k)-p(t_{k-1})}{\Delta t} - \frac{1}{\Delta t}r_1(t_k) = D_tp(t_k) - \frac{1}{\Delta t}r_1(t_k),
\end{equation}
We have 
$$
\begin{array}{rcl}
\dsp \| \pi_h \dot p(t_k)- D_t p_h^k \| &\leq & \dsp \| \pi_h \dot p(t_k)- \pi_h  D_t p(t_k) \| +\| D_t  (\pi_h p(t_k) - p_h^k)\| \\
 & \leq & \dsp   \frac{1}{\Delta t}\|r_1(t_k)\| + \| D_t  (\pi_h p(t_k) - p_h^k) \|
\end{array}
$$
Therefore, the error we need to bound satisfies 
\begin{equation}
\label{eq-ErreurEnergie-wave}
\| \pi_hp(t_k)-p_h^k \|_\frac{1}{2} + \| \pi_h \dot p(t_k)- D_t p_h^k \| \leq
 2 \sqrt{\mathcal{E}_{h}^k} + \frac{1}{\Delta t}\|r_1(t_k)\|
\end{equation}
where we have set for all $1\le k \le K$
$$
\mathcal{E}_{h}^k = \frac{1}{2}\left\{ \left\|D_t \left(\pi_h p(t_k)- p_h^k\right) \right\|^2 + \left\|\pi_h p(t_k)-p_{h}^k\right\|_\frac{1}{2}^2\right\}.
$$
On the other hand, if $r_2(t_k)$ denote the residual term first order the Taylor expansion of $\dot p$ around $t_{k-1}$, then
$$
\begin{array}{rcl}
\dsp\ddot{q}(t_k) &=& \dsp\frac{\dot p(t_k) - \dot p(t_{k-1})}{\Delta t} - \frac{1}{\Delta t}r_2(t_k), \\
&=& \dsp\frac{p(t_k) - 2 p(t_{k-1}) + p(t_{k-2})}{\Delta t^2} - \frac{1}{\Delta t^2}\left(r_1(t_{k})-r_1(t_{k-1})\right) - \frac{1}{\Delta t}r_2(t_k),
\end{array}
$$
yielding
\begin{equation}
\label{eq-taylor-dot-wave}
\ddot{p}(t_k) = D_{tt} p(t_k) - \gamma^k, 
\end{equation}
where $$\gamma^k = \frac{1}{\Delta t^2}\left(r_1(t_k)-r_1(t_{k-1})\right) + \frac{1}{\Delta t}r_2(t_k).$$
Using \eqref{eq-taylor-wave} and \eqref{eq-taylor-dot-wave}, the variational formulation \eqref{eq-general-q-FV-wave} written for $t=t_k$ and for an arbitrary test function $\varphi = \varphi_h\in H_h$ takes the form
$$
\left\langle D_{tt} p(t_k)-\gamma^k,\varphi_h \right\rangle + \left\langle \pi_h p(t_k),\varphi_h \right\rangle_\frac{1}{2} 
+ \Big\langle C_0^*C_0 \Big(D_t p(t_k) -  \frac{1}{\Delta t}r_1(t_k)\Big),\varphi_h \Big\rangle = \left\langle f(t_k),\varphi_h \right\rangle.
$$
Subtracting \eqref{eq-general-qhk-FVDF-wave} from the above relation implies that
for all $2\le k \le K$ and all $\varphi_h\in H_h$:
\begin{multline*}
\left\langle D_{tt} ( p(t_k)-p_h^k),\varphi_h \right\rangle + \left\langle \pi_h p(t_k) -p_h^k,\varphi_h \right\rangle_\frac{1}{2} = - \left\langle C_0^*C_0 D_t \left(p(t_k) - p_h^k\right),\varphi_h \right\rangle \\
+ \left\langle \gamma^k,\varphi_h \right\rangle 
+\frac{1}{\Delta t}  \left\langle C_0^*C_0 r_1(t_k),\varphi_h \right\rangle
+ \left\langle f(t_k)-f_h^k,\varphi_h \right\rangle.
\end{multline*}
From the above relation, we get that
\begin{multline}\label{eq-tech-pihq-qhk-wave}
\left\langle D_{tt} \left(\pi_h p(t_k)-p_h^k\right),\varphi_h \right\rangle + \left\langle \pi_h p(t_k) -p_h^k,\varphi_h \right\rangle_\frac{1}{2} = \left\langle D_{tt}\left(\pi_hp(t_k)-p(t_k)\right),\varphi_h \right\rangle \\
- \left\langle C_0^*C_0 D_t \left(p(t_k) - p_h^k\right),\varphi_h \right\rangle 
+ \left\langle \gamma^k,\varphi_h \right\rangle 
+\frac{1}{\Delta t}  \left\langle C_0^*C_0 r_1(t_k),\varphi_h \right\rangle
+ \left\langle f(t_k)-f_h^k,\varphi_h \right\rangle.
\end{multline}
Using the identity 
$$
\frac{1}{2} \left( \|u\|^2 - \|v\|^2 + \|u-v\|^2 \right) = \mbox{Re} \, \left\langle u-v,u \right\rangle, \quad \forall u,v \in H,
$$
one easily obtains that for all $2\le k \le K$
\begin{multline*}
D_t\mathcal{E}_{h}^k \leq \left\langle D_{tt}\left(\pi_h p(t_k)-p_{h}^k\right),D_t\left(\pi_h p(t_k)-p_h^k\right) \right\rangle \\
+ \left\langle \pi_h p(t_k)-p_{h}^k,D_t\left(\pi_h p(t_k)-p_h^k\right) \right\rangle_\frac{1}{2}.
\end{multline*}
Taking $\varphi_h = D_t\left(\pi_h p(t_k)-p_h^k\right)$ in \eqref{eq-tech-pihq-qhk-wave} 
and substituting in the above inequality and using the boundedness of $C_0$, we obtain the existence of $M>0$ such that for all $2\le k \le K$
\begin{multline}\label{eq-intermediate-full-1-wave}
D_t\mathcal{E}_{h}^k \leq M\bigg[\left\| D_{tt}\left(\pi_hp(t_k)-p(t_k)\right) \right\| + \left\| D_t \left( \pi_h p(t_k) - p(t_k)\right) \right\| + \| \gamma^k\| \\
+ \frac{1}{\Delta t}\| r_1(t_k)\|+\| f(t_k)-f_h^k\|\bigg]\left\| D_t(\pi_hp(t_k)-p_h^k) \right\|.
\end{multline}
Using relations \eqref{eq-DtId-1} and \eqref{eq-DtId-2}, 
we obtain from \eqref{eq-est-H} and \eqref{eq-intermediate-full-1-wave} that for all $h\in(0,h^*)$ 
$$
D_t\sqrt{\mathcal{E}_{h}^k} \leq M \bigg\{ h^\theta \left( \|D_{tt}p(t_k)\|_\frac{1}{2} + \|D_t p(t_k)\|_\frac{1}{2} \right) + \|\gamma^k\| + \frac{1}{\Delta t}\| r_1(t_k)\|+\|f(t_k)-f_h^k\| \bigg\}.
$$
By \eqref{eq-taylor-wave}, \eqref{eq-taylor-dot-wave} and relations \eqref{eq-q-norm-1-2} and \eqref{eq-dotq-norm-1} in Lemma \ref{lem-annexe2} of the Appendix for the first order formulation of \eqref{eq-general-q-wave}, the last estimate yields
\begin{multline}\label{eq-intermediate-full-2-wave}
D_t\sqrt{\mathcal{E}_{h}^k} \leq M \bigg\{h^\theta \left( \|p_0\|_\frac{3}{2} + \|p_1\|_1 + t_k\|f\|_{1,\infty} + \|f\|_{\frac{1}{2},\infty} \right) +\|f(t_k)-f_h^k\| \\
+ \frac{h^\theta}{\Delta t^2}\|r_1(t_k)-r_1(t_{k-1})\|_\frac{1}{2} + \frac{h^\theta}{\Delta t}\left(\|r_1(t_k)\|_\frac{1}{2} + \|r_2(t_k)\|_\frac{1}{2}\right) \\
+ \frac{1}{\Delta t^2}\|r_1(t_k)-r_1(t_{k-1})\| + \frac{1}{\Delta t} \Big(\|r_1(t_k)\| + \|r_2(t_k)\| \Big) \bigg\}.
\end{multline}
To conclude, it remains to bound the terms including the residuals $r_1$ and $r_2$ in the above estimate. By definition of $r_2$, we have
$$
r_2(t_k) = \dot p(t_k)-\dot p(t_{k-1}) - \Delta t \,\ddot{p}(t_k),
$$
in $\Dzdemi$, and thus by the mean value theorem, we get
$$
\|r_2(t_k)\|_\frac{1}{2} \leq \Delta t \sup_{s\in\left[t_{k-1},t_k\right]}\|\ddot{p}(s)\|_\frac{1}{2} + \Delta t\|\ddot{p}(t_k)\|_\frac{1}{2}.
$$
Using once again \eqref{eq-dotq-norm-1}, we obtain that there exists $M>0$ such that
\begin{equation}\label{eq-intermediate-full-3-wave}
\|r_2(t_k)\|_\frac{1}{2} \leq M \Delta t \left( \|p_0\|_\frac{3}{2} + \|p_1\|_1 + t_k\|f\|_{1,\infty} + \|f\|_{\frac{1}{2},\infty} \right).
\end{equation}
Now by the regularity of $p$ (see Lemma \ref{lem-annexe2} applied to the first order formulation of \eqref{eq-general-q-wave}), the residual $r_2$ can be expressed via the integral 
$$
r_2(t_k) = \int_{t_{k-1}}^{t_k} \frac{d^3p}{ds^3}(s)\left(t_{k-1}-s\right)ds,
$$
in $H$, and thus 
$$
\|r_2(t_k)\| \leq \Delta t^2 \sup_{s\in\left[t_{k-1},t_k\right]}\left\| \frac{d^3p}{ds^3}(s)\right\|.
$$
Using equation \eqref{eq-general-q-wave} verified by $p$ and the boundedness of $C_0$, we have
\begin{equation}\label{eq-intermediate-full-8-wave}
\begin{array}{lll}
\left\| \dfrac{d^3p}{dt^3}(t)\right\| = \Big\|\dfrac{d\ddot{p}}{dt}(t) \Big\|&= \Big\|\dfrac{d}{dt}\Big\{ -A_0 p(t) -C_0^*C_0\dot p(t) + f(t) \Big\} \Big\|, \\
&\leq \|\dot{p}(t)\|_1 + M \|\ddot{p}(t)\| + \|\dot f(t)\|.
\end{array}
\end{equation}
Hence, once again by \eqref{eq-dotq-norm-1}, we get 
\begin{equation}\label{eq-intermediate-full-4-wave}
\|r_2(t_k)\| \leq M \Delta t^2 \left( \|p_0\|_\frac{3}{2} + \|p_1\|_1 + t_k\|f\|_{1,\infty} + \|f\|_{\frac{1}{2},\infty} + \|\dot f\|_\infty\right).
\end{equation}
For the term implying $r_1$, we note that
$$
r_1(t_k) = \int_{t_{k-1}}^{t_k} \ddot p(s) (t_{k-1}-s) ds,
$$
in $\Dzdemi$. Hence, by a similar argument and \eqref{eq-dotq-norm-1},
\begin{equation}\label{eq-intermediate-full-5-wave}
\|r_1(t_k)\| \leq M \|r_1(t_k)\|_\frac{1}{2} \leq M \Delta t^2 \left( \|p_0\|_\frac{3}{2} + \|p_1\|_1 + t_k\|f\|_{1,\infty} + \|f\|_{\frac{1}{2},\infty} \right).
\end{equation}

Then, we write in $\Dzdemi$
$$
r_1(t_k)-r_1(t_{k-1}) = \int_{t_{k-2}}^{t_{k-1}} \left(\ddot p(s-\Delta t)-\ddot p(s)\right)\left(t_{k-2}-s\right)ds.
$$
Using the above relation, it comes by using once again \eqref{eq-dotq-norm-1}
\begin{eqnarray}
\|r_1(t_k)-r_1(t_{k-1})\|_\frac{1}{2} &\leq& M \Delta t^2 \sup_{s \in (t_{k-2},t_{k-1})} \|\ddot p(s) \|_\frac{1}{2}, \nonumber \\
&\leq& M \Delta t^2 \left( \|p_0\|_\frac{3}{2} + \|p_1\|_1 + t_{k-1}\|f\|_{1,\infty} + \|f\|_{\frac{1}{2},\infty} \right).\label{eq-intermediate-full-6-wave}
\end{eqnarray}
Finally 
\begin{eqnarray*}
\|r_1(t_k)-r_1(t_{k-1})\| &\leq& \Delta t \int_{t_{k-2}}^{t_{k-1}} \int_{s-\Delta t}^s \left\|\dfrac{d^3p}{d\sigma^3}(\sigma) \right\| d\sigma \, ds, \\
&\leq& M \Delta t^3 \sup_{s \in (t_{k-3},t_{k-1})} \left\|\dfrac{d^3p}{ds^3}(s) \right\| .
\end{eqnarray*}
Using \eqref{eq-intermediate-full-8-wave} and \eqref{eq-dotq-norm-1}, we get
\begin{equation}\label{eq-intermediate-full-7-wave}
\|r_1(t_k)-r_1(t_{k-1})\| \leq  M \Delta t^3 \left( \|p_0\|_\frac{3}{2} + \|p_1\|_1 + t_{k-1}\|f\|_{1,\infty} + \|f\|_{\frac{1}{2},\infty} + \|\dot f\|_\infty\right).
\end{equation}
Substituting \eqref{eq-intermediate-full-3-wave}, \eqref{eq-intermediate-full-4-wave}, \eqref{eq-intermediate-full-5-wave}, \eqref{eq-intermediate-full-6-wave} and \eqref{eq-intermediate-full-7-wave} in relation \eqref{eq-intermediate-full-2-wave} provides estimates for $D_t\sqrt{\mathcal{E}_{h}^k}=\dfrac{\sqrt{\mathcal{E}_{h}^k}-\sqrt{\mathcal{E}_{h}^{k-1}}}{\Delta t}$, for $k=1,\dots,K$. By adding all these inequalities, we immediately get an upper bound for $\sqrt{\mathcal{E}_{h}^k}$, and thus the desired inequality thanks to \eqref{eq-ErreurEnergie-wave} and \eqref{eq-intermediate-full-5-wave}.
\end{proof}
Using this Proposition, we can derive an error estimate for the semigroup $\T^\pm_{t_k}$ (for all $0\le k\le K$) and for the operator $\Lb_\tau = \T^-_\tau\T^+_\tau$ (the counterpart of Proposition \ref{prop-L-tau-h-wave}).
\begin{proposition}\label{prop-L-tau-hk-wave}
Let $\Pi_h = \left[\begin{matrix}\pi_h & 0\\0 & \pi_h\end{matrix}\right]$. Under the assumptions of Proposition \ref{prop-pihq-qhk-wave}, the following assertions hold true
\begin{enumerate}
\item There exist $M>0$, $h^*>0$ and $\Delta t^*>0$ such that for all $h\in(0,h^*)$, all $\Delta t\in(0,\Delta t^*)$ and all $0\le k\le K$
\begin{equation}\label{eq-cor-pihq-qhk-forward-wave}
\left\|(\Pi_h\T^+_{t_k}  - \T^+_{h,\Delta t, k}) \Pz \right\| \leq M t_k(h^\theta + \Delta t) \left( \|p_0\|_\frac{3}{2} + \|p_1\|_1 \right).
\end{equation}
\begin{equation}\label{eq-cor-pihq-qhk-backward-wave}
\left\|(\Pi_h\T^-_{t_k}  - \T^-_{h,\Delta t, k}) \Pz \right\| \leq M (\tau-t_k)(h^\theta + \Delta t) \left( \|p_0\|_\frac{3}{2} + \|p_1\|_1 \right).
\end{equation}
\item There exist $M>0$, $h^*>0$ and $\Delta t^*>0$ such that for all $n\in\mathbb{N}$, all $h\in(0,h^*)$, all $\Delta t\in(0,\Delta t^*)$ and all $0\le k\le K$
\begin{equation}\label{eq-cor-pihq-qhk-wave2}
\left\| (\Lb_{t_k}^n - \Lb_{h,\Delta t,k}^n) \Pz \right\| \leq M \left[ h^\theta + n \tau(h^\theta+\Delta t) \right] \left( \|p_0\|_\frac{3}{2} + \|p_1\|_1 \right).
\end{equation}

\end{enumerate}
\end{proposition}
\begin{proof}\ 

1. From relations \eqref{eq-Tpm-h-wave} and \eqref{eq-Tpm-h-Deltat-wave} defining the continuous and the fully discretized operators $\T^\pm_{t_k}$ and $\T^\pm_{h,\Delta t, k}$, \eqref{eq-cor-pihq-qhk-forward-wave} and \eqref{eq-cor-pihq-qhk-backward-wave}  follow immediately from Proposition \ref{prop-pihq-qhk-wave}. 


2. First, we note that 
\begin{equation}\label{eq-Lt-ineq-tri-full-wave}
\left\| (\Lb^n_{t_k} - \Lb^n_{h,\Delta t,k}) \Pz \right\| \leq \left\| (\Lb^n_{t_k} - \Pi_h\Lb^n_{t_k})\Pz \right\| + \left\| (\Pi_h\Lb^n_{t_k}  - \Lb^n_{h,\Delta t, k} )\Pz \right\|.
\end{equation}
Using \eqref{eq-est-H}, the fact that $\|\Lb^n_t\|_{\mathcal{L}({\cal D}(A))} \leq 1$ (proved in Lemma \ref{lem-contraction} of the Appendix), the first term in the above relation can be estimated as follows 
\begin{equation}\label{eq-Lt-first-term-full-wave}
\left\| (\Lb^n_{t_k} - \Pi_h\Lb^n_{t_k})\Pz \right\| \leq Mh^\theta\left( \|p_0\|_\frac{3}{2} + \|p_1\|_1 \right), \quad \forall h \in (0,h^*).
\end{equation}
For the second term in \eqref{eq-Lt-ineq-tri-full-wave}, we prove by induction that for all $n\in\mathbb{N}$, all $h \in (0,h^*)$ and all $\Delta t\in(0,\Delta t^*)$ (for some $\Delta t^*>0$)
\begin{equation}\label{eq-Lt-second-term-full-wave}
\left\| (\Pi_h\Lb^n_{t_k} - \Lb^n_{h,\Delta t, k}) \Pz \right\| \leq M n t_k(h^\theta + \Delta t) \left( \|p_0\|_\frac{3}{2} + \|p_1\|_1 \right).
\end{equation}
By definition, we have 
$$
\begin{array}{lll}
\left\| (\Pi_h \Lb_{t_k} - \Lb_{h,\Delta t, k} )\Pz \right\| &= \left\| (\Pi_h \T^-_{t_k} \T^+_{t_k} - \T_{h,\Delta t, k}^- \T_{h,\Delta t, k}^+) \Pz \right\|, \\
&\leq \left\|(\Pi_h \T^-_{t_k} - \T_{h,\Delta t, k}^-) \Pi_h \T_{t_k}^+ \Pz \right\|\\
&\hspace{1cm} + \left\| \T_{h,\Delta t, k}^-( \Pi_h \T_{t_k}^+ - \T_{h,\Delta t, k}^+) \Pz \right\|.
\end{array}
$$
Using \eqref{eq-cor-pihq-qhk-backward-wave} and Lemma \ref{lem-contraction}, we get 
$$
\left\|( \Pi_h \T^-_{t_k} - \T_{h,\Delta t, k}^-) \Pi_h \T_{t_k}^+ \Pz \right\| \leq M (\tau-t_k)( h^\theta + \Delta t ) \left( \|p_0\|_\frac{3}{2} + \|p_1\|_1 \right).
$$
Obviously $\| \T^-_{h,\Delta t,k} \|_{\mathcal{L}(X)}$ is uniformly bounded (with respect to $h$ and $\Delta t$), and thus again by \eqref{eq-cor-pihq-qhk-forward-wave} we have
$$
\left\| \T_{h,\Delta t,k}^-( \Pi_h \T_{t_k}^+ - \T_{h,\Delta t,k}^+) \Pz \right\| \leq M t_k( h^\theta + \Delta t) \left( \|p_0\|_\frac{3}{2} + \|p_1\|_1 \right).
$$
So
\begin{equation}\label{eq-pih-Lt-full-wave}
\left\| (\Pi_h\Lb_{t_k} - \Lb_{h,\Delta t,k}) \Pz \right\| \leq M \tau ( h^\theta + \Delta t) \left( \|p_0\|_\frac{3}{2} + \|p_1\|_1 \right),
\end{equation}
showing that \eqref{eq-Lt-second-term-full-wave} holds for $n=1$. Suppose now that for some $n\geq 2$ 
\begin{equation}\label{eq-pih-Lt-rec-full-wave}
\left\| (\Pi_h\Lb^{n-1}_{t_k} - \Lb^{n-1}_{h,\Delta t,k} ) \Pz \right\| \leq M (n-1) \tau( h^\theta + \Delta t) \left( \|p_0\|_\frac{3}{2} + \|p_1\|_1 \right).
\end{equation}
Writing 
\begin{multline*}
\left\| (\Pi_h\Lb^n_{t_k} - \Lb^n_{h,\Delta t,k}) \Pz \right\| \leq \left\| (\Pi_h\Lb_{t_k}\Lb^{n-1}_{t_k} - \Lb_{h,\Delta t,k}\Pi_h\Lb^{n-1}_{t_k}) \Pz \right\|\\
+ \left\| \Lb_{h,\Delta t,k}(\Pi_h\Lb^{n-1}_{t_k} p_0 - \Lb^{n-1}_{h,\Delta t,k} p_0) \right\|,
\end{multline*}
we get by using Lemma \ref{lem-contraction}, the uniform boundedness of $\|\Lb_{h,\Delta t,k}\|_{\mathcal{L}(X)}$ with respect to $h$ and $\Delta t$, \eqref{eq-pih-Lt-full-wave} and \eqref{eq-pih-Lt-rec-full-wave} that
$$
\left\| (\Pi_h\Lb^n_{t_k} - \Lb^n_{h,\Delta t,k} )\Pz \right\| \leq M \left[(1+(n-1)) \tau( h^\theta + \Delta t) \right]\left( \|p_0\|_\frac{3}{2} + \|p_1\|_1 \right),
$$
which is exactly \eqref{eq-Lt-second-term-full-wave}. Substituting \eqref{eq-Lt-first-term-full-wave} and \eqref{eq-Lt-second-term-full-wave} in \eqref{eq-Lt-ineq-tri-full-wave}, we obtain the result.
\end{proof}
We are now able to prove Theorem \ref{th-main-full-wave}.
\begin{proof}[of Theorem \ref{th-main-full-wave}]
Introducing the term $\dsp \sum_{n=0}^{N_{h,\Delta t}} \Lb_{h,\Delta t,K}^n \left[\begin{matrix}w^-(0) \\ \dot w^-(0)\end{matrix}\right] $, we can rewrite $\left[\begin{matrix}w_0 \\ w_1\end{matrix}\right] - \left[\begin{matrix}w_{0,h,\Delta t} \\ w_{1,h,\Delta t}\end{matrix}\right]$ in the following form
$$
\begin{array}{lll}
\left[\begin{matrix}w_0 \\ w_1\end{matrix}\right] - \left[\begin{matrix}w_{0,h,\Delta t} \\ w_{1,h,\Delta t}\end{matrix}\right] &= \dsp \sum_{n=0}^\infty \Lb_\tau^n \left[\begin{matrix}w^-(0) \\ \dot w^-(0)\end{matrix}\right] -\sum_{n=0}^{N_{h,\Delta t}} \Lb_{h,\Delta t,K}^n \left[\begin{matrix}(w_h^-)^0 \\ D_t (w_h^-)^1\end{matrix}\right],\\
&= \dsp \sum_{n>N_{h,\Delta t}} \Lb_\tau^n \left[\begin{matrix}w^-(0) \\ \dot w^-(0)\end{matrix}\right] + \sum_{n=0}^{N_{h,\Delta t}}\left(\Lb_\tau^n - \Lb_{h,\Delta t,K}^n\right) \left[\begin{matrix}w^-(0) \\ \dot w^-(0)\end{matrix}\right] \\
& \dsp \hspace{3cm} + \sum_{n=0}^{N_{h,\Delta t}} \Lb_{h,\Delta t,K}^n \left(\left[\begin{matrix}w^-(0)-(w_h^-)^0 \\ \dot w^-(0) - D_t (w_h^-)^1\end{matrix}\right]\right).
\end{array} 
$$
Therefore, we have 
\begin{equation}\label{eq-diff-z0-z0hk-wave}
\| w_0 - w_{0,h,\Delta t} \|_\frac{1}{2} + \| w_1 - w_{1,h,\Delta t} \| \leq S_1+S_2+S_3,
\end{equation}
where we have set 
$$
\left\{
\begin{array}{lll}
\dsp S_1= \sum_{n>N_{h,\Delta t}} \left\| \Lb_\tau^n \left[\begin{matrix}w^-(0) \\ \dot w^-(0)\end{matrix}\right] \right\|,\\
\dsp S_2=\sum_{n=0}^{N_{h,\Delta t}} \left\| \left(\Lb_\tau^n - \Lb_{h,\Delta t,K}^n\right) \left[\begin{matrix}w^-(0) \\ \dot w^-(0)\end{matrix}\right] \right\|,\\
\dsp S_3=\left(\sum_{n=0}^{N_{h,\Delta t,K}} \left\| \Lb_{h,\Delta t,K}^n \right\|_{\mathcal{L}(X)}\right) \left\|\left[\begin{matrix}w^-(0)-(w_h^-)^0 \\ \dot w^-(0) - D_t (w_h^-)^1\end{matrix}\right]\right\|.
\end{array} 
\right.
$$
Note that the term $S_1$ is the truncation error of the tail of the infinite sum \eqref{eq-z0-neumann}, the term $S_2$ represents the cumulated error due to the approximation of the semigroups $\T^\pm$ while the term $S_3$ comes from the approximation of the first iterate $\left[\begin{matrix}w^-(0) \\ \dot w^-(0)\end{matrix}\right]$ of the algorithm. 

Since $\eta=\|\Lb_\tau\|_{\mathcal{L}(X)} < 1$, using relation \eqref{eq-z0-Ltau}, the first term can be estimated very easily 
\begin{equation}\label{eq-first-term-th-full-wave}
S_1\leq M\, \frac{\eta^{N_{h+\Delta t}+1}}{1-\eta} \left( \| w_0 \|_\frac{3}{2} + \| w_1 \|_1 \right).
\end{equation}
The term $S_2$ can be estimated using the estimate \eqref{eq-cor-pihq-qhk-wave2} from Proposition \ref{prop-L-tau-hk-wave} : for all $h \in (0,h^*)$ and all $\Delta t\in (0,\Delta t^*)$
$$
S_2 \leq M \left(\sum_{n=0}^{N_{h,\Delta t}} (1+n \tau)\right) \left(h^\theta + \Delta t\right) \left( \|w^-(0)\|_\frac{3}{2} + \|\dot w^-(0)\|_1 \right),
$$
Therefore, using \eqref{eq-z0-Ltau} and the fact that $\|\Lb_\tau\|_{\Ds}<1$ (see Lemma \ref{lem-contraction}) in the above relation, we finally get that for all $h \in (0,h^*)$ and all $\Delta t\in (0,\Delta t^*)$
\begin{equation}\label{eq-second-term-th-full-wave}
S_2 \leq M\Big[1+(1+\tau)N_{h,\Delta t}+N_{h,\Delta t}^2\tau \Big] \left( h^\theta + \Delta t \right) \left( \| w_0 \|_\frac{3}{2} + \| w_1 \|_1 \right),
\end{equation}
It remains to estimate the term $S_3$. By the uniform boundedness of $\|\Lb_{h,\Delta t, K}\|_{\mathcal{L}(X)}$ by 1 with respect to $h$ and $\Delta t^*$, we have 
\begin{equation}\label{eq-S3-num1-wave-full}
\begin{array}{rcl}
S_3 &\leq & \dsp M N_{h,\Delta t} \left(\| w^-(0) - (w_h^-)^0 \|_\frac{1}{2} + \| \dot w^-(0) - D_t (w_h^-)^1 \| \right)\\
 & \leq &MN_{h,\Delta t} \dsp\Big(\left\| w^-(0) - \pi_h w^-(0)\right\|_\frac{1}{2} + \left\|\pi_h w^-(0)- (w_h^-)^0\right\|_\frac{1}{2} \\
& &\hspace{2.5cm} + \left\| \dot w^-(0) - \pi_h \dot w^-(0)\right\| + \left\|\pi_h \dot w^-(0)- D_t (w_h^-)^1\right\| \Big).
\end{array}
\end{equation}
By using \eqref{eq-est-H} and \eqref{eq-z0-Ltau}, we immediately obtain that
\begin{equation}\label{eq-S3-num2-wave-full}
\| w^-(0) - \pi_h w^-(0) \|_\frac{1}{2} + \| \dot w^-(0) - \pi_h \dot w^-(0) \| \leq M\left(h^\theta +\Delta t\right)\left( \| w_0 \|_\frac{3}{2} + \| w_1 \|_1 \right).
\end{equation}
To estimate $\| \pi_h w^-(0)- (w_h^-)^0 \|_\frac{1}{2} + \| \pi_h \dot w^-(0) - D_t (w_h^-)^1 \|$, we apply twice Proposition \ref{prop-pihq-qhk-wave} first for the time reversed backward observer $w^-(\tau-\cdot)$ and then for the forward observer $w^+$ (the time reversal is introduced because Proposition \ref{prop-pihq-qhk-wave}, as it is formulated, concerns initial value Cauchy problems). After straightforward calculation we obtain that for all $h\in(0,h^*)$ and all $\Delta t \in (0,\Delta t^*)$
\begin{multline}\label{eq-intermediate-main-1-full-wave}
\| \pi_h w^-(0) - (w_h^-)^0 \|_\frac{1}{2} + \| \pi_h \dot w^-(0) - D_t (w_h^-)^1 \| \\
 \leq M \left(h^\theta+\Delta t\right)\Big[ \tau(\|w^+(\tau)\|_\frac{3}{2} + \| \dot w^+(\tau) \|_1 + \|C_0^*y\|_{\frac{1}{2}, \infty}) + \tau^2\|C_0^*y\|_{1,\infty} \Big]\\
+ \Delta t \sum_{\ell=1}^K \|C_0^*\left(y(\tau-\ell)-y_h^{K-\ell} \right)\| + \Delta t \sum_{\ell=1}^K \|C_0^*\left(y(t_\ell)-y_h^\ell \right)\|.
\end{multline}
Applying \eqref{eq-q-norm-1-2} of Lemma \ref{lem-annexe2} of the Appendix with zero initial data, we obtain that 
$$
\|w^+(\tau)\|_\frac{3}{2} + \|\dot w^+(\tau)\|_1 \leq \tau \|C_0^*y\|_{1,\infty}.
$$
Therefore \eqref{eq-intermediate-main-1-full-wave} also reads
\begin{multline*}
\| \pi_h w^-(0) - (w_h^-)^0 \|_\frac{1}{2} + \| \pi_h \dot w^-(0) - D_t (w_h^-)^1 \| \leq M \left( h^\theta + \Delta t \right) (\tau + \tau^2)\|C_0^*y\|_{1,\infty} \\
+ 2 \Delta t \sum_{\ell=0}^K\|C_0^*\left(y(t_\ell)-y_h^\ell\right)\|.
\end{multline*}
As $C_0^*C_0 \in \mathcal{L}\left(\Dztdemi\right) \cap \mathcal{L}\left(\Dzdemi\right)$ and $\|w\|_{\frac{3}{2},\infty} + \|\dot w\|_{1,\infty} = \|w_0\|_\frac{3}{2} + \|w_1\|_1$ (since $A$ is skew-adjoint), the last relation becomes 
\begin{multline*}
\| \pi_h w^-(0) - (w_h^-)^0 \|_\frac{1}{2} + \| \pi_h \dot w^-(0) - D_t (w_h^-)^1 \| \\
\leq M \left( h^\theta + \Delta t \right) (\tau + \tau^2)\left(\|w_0\|_\frac{3}{2} + \|w_1\|_1\right) + 2 \Delta t \sum_{\ell=0}^K\|C_0^*\left(y(t_\ell)-y_h^\ell\right)\|.
\end{multline*}
Substituting the above relation and \eqref{eq-S3-num2-wave-full} in \eqref{eq-S3-num1-wave-full}, we get
\begin{multline}\label{eq-third-term-th-full-wave}
S_3 \leq M N_{h,\Delta t} \bigg[ \left(h^\theta + \Delta t\right) (1+\tau + \tau^2)\left(\|w_0\|_\frac{3}{2} + \|w_1\|_1\right) \\
+ \Delta t \sum_{\ell=0}^K\|C_0^*\left(y(t_\ell)-y_h^\ell\right)\| \bigg].
\end{multline}
Substituting \eqref{eq-first-term-th-full-wave}, \eqref{eq-second-term-th-full-wave} and \eqref{eq-third-term-th-full-wave} in \eqref{eq-diff-z0-z0hk-wave}, we get for all $h\in(0,h^*)$ and all $\Delta t\in(0,\Delta t^*)$
\begin{multline*}
\| w_0 - w_{0,h,\Delta t} \|_\frac{1}{2} + \| w_1 - w_{1,h,\Delta t} \| \leq \\
M \Bigg[ \bigg(\frac{\eta^{N_{h,\Delta t}+1}}{1-\eta} + \left(h^\theta + \Delta t\right) \left[1+(1+\tau+\tau^2)N_h+\tau N_h^2 \right] \bigg)\left(\|w_0\|_\frac{3}{2} + \|w_1\|_1\right) \\
+ N_{h,\Delta t} \Delta t \sum_{\ell=0}^K\|C_0^*\left(y(t_\ell)-y_h^\ell\right)\| \Bigg],
\end{multline*}
which leads to the result (with possibly reducing the value of $h^*$ and $\Delta t^*$).
\end{proof}

\section*{Appendix}
\setcounter{equation}{0}
\setcounter{theorem}{0}
\setcounter{section}{4}
Let $A :\D \rightarrow X$ a skew-adjoint operator and $C \in \mathcal{L}(X,Y)$ such that $C^*C \in \mathcal{L}\left(\D\right)$. Assume that $A-C^*C$ generates a $C_0$-semigroup of contractions on $X$.
\begin{lemma}\label{lem-contraction}
The operator $A-C^*C$ generates a $C_0$-semigroup of contractions on $\D$ and $\Ds$.
\end{lemma}
\begin{proof}
As $C\in \mathcal{L}(X,Y)$ is bounded, we clearly have $\D=\mathcal{D}\left(A-C^*C\right)$. Moreover, $C^*C \in \mathcal{L}\left(\D\right)$ implies that $\Ds=\mathcal{D}\left(\left(A-C^*C\right)^2\right)$. The result follows then from \cite[Proposition 2.10.4]{TucWei09}. 
\end{proof}
\begin{lemma}\label{lem-annexe2}
Given $q_0 \in \Ds$ and $F \in C\left([0,\tau],\Ds\right) \cap C^1\left([0,\tau],\D\right)$, let $q$ denote the solution of the initial value problem
$$
\left\{\begin{array}{ll}
\dot{q}(t) &= A q(t) - C^*C q(t) + F(t), \quad t \in (0,\tau), \\
q(0) &= q_0.
\end{array}\right.
$$
Then, we have the following statements
\begin{enumerate}
\item Regularity:
\begin{equation}\label{eq-reg-q}
q \in C\left([0,\tau],\Ds\right) \cap C^1\left([0,\tau],\D\right) \cap C^2\left([0,\tau],X\right),
\end{equation}
\item Bound for $q$:
\begin{equation}\label{eq-q-norm-1-2}
\| q(t) \|_\alpha \leq \| q_0 \|_\alpha + t \| F \|_{\alpha,\infty}, \qquad  \mbox{ for } \alpha = 0, 1, 2,
\end{equation}
\item Bound for $\dot q$ : there exists $M>0$ such that
\begin{equation}\label{eq-dotq-norm-1}
\| \dot{q}(t) \|_\alpha \leq M \left( \| q_0 \|_{\alpha+1} + t \| F \|_{\alpha+1,\infty} \right) + \| F \|_{\alpha,\infty}, \qquad \mbox{ for } \alpha = 0, 1,
\end{equation}
\end{enumerate}
where $\dsp \| F \|_{\alpha,\infty} = \sup_{t\in[0,\tau]} \| F(t) \|_\alpha$.
\end{lemma}
\begin{proof} \ 

1. By \cite[Theorem 4.1.6]{TucWei09}, we have $q \in C\left([0,\tau],\Ds\right) \cap C^1\left([0,\tau],\D\right)$. But since $C^*C \in \mathcal{L}\left(\D\right)$ and $F \in C\left([0,\tau],\Ds\right) \cap C^1\left([0,\tau],\D\right)$, we have $$\left(A-C^*C\right)q(t) \in C\left([0,\tau],\D\right) \cap C^1\left([0,\tau],X\right).$$
The last inclusion follows then from the fact that $\dot{q}(t) = \left(A-C^*C\right)q(t)$ in $\D$.

2. By Duhamel's formula, we have 
$$
\begin{array}{lll}
\| q(t) \|_\alpha & =\dsp \Big\| \T_tq_0 + \int_0^t \T_{t-s} F(s) ds \Big\|_\alpha, \\
& \dsp\leq \left\| \T_tq_0 \right\|_\alpha + \int_0^t \left\| \T_{t-s} F(s) \right\|_\alpha ds, \\
& \dsp \leq \| q_0 \|_\alpha + t \| F \|_{\alpha,\infty},
\end{array}
$$
where we have used Lemma \ref{lem-contraction} for the last inequality.

3. Using the estimate \eqref{eq-q-norm-1-2} obtained for $q(t)$ and the continuity of the embeddings $\Ds \hookrightarrow \D \hookrightarrow X$, we easily get
$$
\begin{array}{lll}
\| \dot{q}(t) \|_\alpha & = \| \left(A-C^*C\right)q(t) + F(t) \|_\alpha, \\
& \leq \| q(t) \|_{\alpha+1} + M \| q(t) \|_\alpha + \|F \|_{\alpha,\infty}, \\
& \leq M \left( \| q_0 \|_{\alpha+1} + t \|F \|_{\alpha+1,\infty} \right) + \| F \|_{\alpha,\infty}.
\end{array}
$$
\end{proof}

\bibliographystyle{siam}      
\bibliography{observers}   

\begin{thebibliography}{10}

\bibitem{AurBlu08}
{\sc D.~Auroux and J.~Blum}, {\em A nudging-based data assimilation method :
  the back and forth nudging (bfn) algorithm}, Nonlin. Proc. Geophys., 15
  (2008).

\bibitem{CinMicTuc10}
{\sc N.~C{\^i}ndea, S.~Micu, and M.~Tucsnak}, {\em An approximation method for
  exact controls of vibrating systems}, Preprint, hal-00430266 (2010).

\bibitem{CurZwa95}
{\sc R.~F. Curtain and H.~Zwart}, {\em An introduction to infinite-dimensional
  linear systems theory}, vol.~21 of Texts in Applied Mathematics,
  Springer-Verlag, New York, 1995.

\bibitem{DauLio85}
{\sc R.~Dautray and J.-L. Lions}, {\em Analyse math\'ematique et calcul
  num\'erique pour les sciences et les techniques. {T}ome 3}, Collection du
  Commissariat \`a l'\'Energie Atomique: S\'erie Scientifique. [Collection of
  the Atomic Energy Commission: Science Series], Masson, Paris, 1985.

\bibitem{DegSalXu06}
{\sc J.~Deguenon, G.~Sallet, and C.-Z. Xu}, {\em Infinite dimensional observers
  for vibrating systems}, in Proc. of IEEE Conf. on Decision and Control, 2006,
  pp.~3979--3983.

\bibitem{FujSuz91}
{\sc H.~Fujita and T.~Suzuki}, {\em Evolution problems}, in Handbook of
  numerical analysis, Vol.\ II, Handb. Numer. Anal., II, North-Holland,
  Amsterdam, 1991, pp.~789--928.

\bibitem{GevKok85}
{\sc T.~Geveci and B.~Kok}, {\em The convergence of {G}alerkin approximation
  schemes for second-order hyperbolic equations with dissipation}, Math. Comp.,
  44 (1985), pp.~379--390, S17--S25.

\bibitem{GuoGuo09}
{\sc B.-Z. Guo and W.~Guo}, {\em The strong stabilization of a one-dimensional
  wave equation by non-collocated dynamic boundary feedback control},
  Automatica, 45 (2009), pp.~790--797.

\bibitem{GuoSha09}
{\sc B.-Z. Guo and Z.-C. Shao}, {\em Stabilization of an abstract second order
  system with application to wave equations under non-collocated control and
  observations}, Systems Control Lett., 58 (2009), pp.~334--341.

\bibitem{KrsMagVaz09}
{\sc M.~Krstic, L.~Magnis, and R.~Vazquez}, {\em Nonlinear control of the
  viscous burgers equation: Trajectory generation, tracking, and observer
  design}, J. Dyn. Sys., Meas., Control, 131 (2009).

\bibitem{KucKun08}
{\sc P.~Kuchment and L.~Kunyansky}, {\em Mathematics of thermoacoustic
  tomography}, European J. Appl. Math., 19 (2008), pp.~191--224.

\bibitem{LasTri00}
{\sc I.~Lasiecka and R.~Triggiani}, {\em Control theory for partial
  differential equations: continuous and approximation theories. {I}}, vol.~74
  of Encyclopedia of Mathematics and its Applications, Cambridge University
  Press, Cambridge, 2000.

\bibitem{Liu97}
{\sc K.~Liu}, {\em Locally distributed control and damping for the conservative
  systems}, SIAM J. Control Optim., 35 (1997), pp.~1574--1590.

\bibitem{Pue09}
{\sc J.-P. Puel}, {\em A nonstandard approach to a data assimilation problem
  and {T}ychonov regularization revisited}, SIAM J. Control Optim., 48 (2009),
  pp.~1089--1111.

\bibitem{RamTucWei10}
{\sc K.~Ramdani, M.~Tucsnak, and G.~Weiss}, {\em Recovering the initial state
  of an infinite-dimensional system using observers}, Automatica,  (2010).
\newblock To appear.

\bibitem{RavTho98}
{\sc P.-A. Raviart and J.-M. Thomas}, {\em Introduction {\`a} l'analyse
  num{\'e}rique des {\'e}quations aux d{\'e}riv{\'e}es partielles}, Dunod,
  Paris, 1998.

\bibitem{SmyKrs05}
{\sc A.~Smyshlyaev and M.~Krstic}, {\em Backstepping observers for a class of
  parabolic {PDE}s}, Systems Control Lett., 54 (2005), pp.~613--625.

\bibitem{TucWei09}
{\sc M.~Tucsnak and G.~Weiss}, {\em Observation and {C}ontrol for {O}perator
  {S}emigroups}, Birk\"auser {A}dvanced {T}exts, Birk\"auser, Basel, 2009.

\bibitem{Zua05}
{\sc E.~Zuazua}, {\em Propagation, observation, and control of waves
  approximated by finite difference methods}, SIAM Review, 47 (2005),
  pp.~197--243.

\end{thebibliography}
\end{document}